\documentclass[a4paper,11pt]{article}
\usepackage{amsmath,amssymb,amsthm,verbatim,ifthen}
\newcommand{\C}{{\mathbb C}}
\newcommand{\K}{{\mathbb K}}
\newcommand{\N}{{\mathbb N}}
\newcommand{\R}{{\mathbb R}}

\newcommand{\Q}{{\mathbb Q}}

\DeclareMathAlphabet{\mathpzc}{OT1}{pzc}{m}{it}

\newcommand{\abs}[2][\empty]{\ifx#1\empty\left|#2\right|%
\else#1\vert #2 #1\vert\fi}
\newcommand{\Ann}{\mathop{\mathrm{Ann}}}
\newcommand{\Cnt}[1]{{\cal C}^{#1}}
\newcommand{\co}[1]{{#1}^{c}}%complement

\newcommand{\defstyle}[1]{{\bf #1}}
\newcommand{\eps}{\varepsilon}

\newcommand{\ideal}{\unlhd}
\newcommand{\idealproper}{\lhd}
\renewcommand{\implies}{\Rightarrow}
\newcommand{\inner}[3][\empty]{\ifx#1\empty\left\langle #2,#3\right\rangle%
\else#1\langle #2,#3 #1\rangle\fi}%optional arg=size
\newcommand{\Inv}{\mathop{\mathrm{Inv}}}%set of invertibility

\newcommand{\kar}[1]{\chi_{#1}}
\newcommand{\Ker}{\mathop{\mathrm{Ker}}}%kernel
\newcommand{\lspan}[2][\empty]{\ifx#1\empty\left\langle#2\right\rangle%
\else#1\langle #2 #1\rangle\fi}%optional arg=size
\newcommand{\Max}{{\mathcal M}}%max ideal space
\newcommand{\norm}[2][\empty]{\ifx#1\empty\left\Vert#2\right\Vert%
\else#1\Vert #2 #1\Vert\fi}%optional arg=size
\newcommand{\rad}[1]{\sqrt{#1}}%radical
\newcommand{\radpart}[1]{{#1}^{\rad{\phantom{.}}}}

\newcommand{\val}[1]{\ifx#1.v\else v(#1)\fi}

\newcommand{\zpart}[1]{{#1}^z}%as in Mason73
\newcommand{\zclosure}[1]{{#1}_z}%as in Mason73

\newcommand{\Gen}{{\mathcal G}}
\newcommand{\GenC}{\widetilde\C}
\newcommand{\GenK}{\widetilde\K}
\newcommand{\GenR}{\widetilde\R}
\newcommand{\Mod}{{{\mathcal E}_M}}
\newcommand{\Null}{{\mathcal N}}
\newcommand{\sharpnorm}[2][\empty]{\abs[#1]{#2}_{\mathrm e}}

\newcommand{\ster}[1]{{{}^* \mskip-1mu #1}}

\newtheorem*{df}{Definition}
\newtheorem{thm}{Theorem}[section]
\newtheorem{prop}[thm]{Proposition}
\newtheorem{lemma}[thm]{Lemma}
\newtheorem*{cor}{Corollary}
\newtheorem{ex}[thm]{Example}

\theoremstyle{remark}%styles: plain, definition, remark
\newtheorem*{rem}{Remark}

\begin{document}
\title{Ideals in the ring of Colombeau generalized numbers}
\author{Hans~Vernaeve\footnote{Supported by FWF (Austria), grants M949-N18 and  Y237-N13.}\\
\\
Institut f\"ur Grundlagen der Bauingenieurwissenschaften\\
Universit\"at Innsbruck\\
Technikerstra{\ss}e 13\\
A-6020 Innsbruck, Austria}
\date{}
\maketitle

\begin{abstract}
In this paper, the structure of the ideals in the ring of Colombeau generalized numbers is investigated. Connections with the theories of exchange rings, Gelfand rings and lattice-ordered rings are given. Characterizations for prime, projective, pure and topologically closed ideals are given, answering in particular the questions about prime ideals in \cite{AJ2001}. Also $z$-ideals in the sense of \cite{Mason73} are characterized. %They appear to be useful to obtain information on the maximal ideals of the algebras of generalized functions.
The quotient rings modulo maximal ideals are shown to be canonically isomorphic with nonstandard fields of asymptotic numbers. Finally, a detailed study of the ideals allows us to prove that (under some set-theoretic assumption) the Hahn-Banach extension property does not hold for a large class of topological modules over the ring of Colombeau generalized numbers.
\end{abstract}

\section{Introduction}
In \cite{AJ2001}, J.~Aragona and S.O.~Juriaans made a detailed study of algebraic properties of the topological ring of Colombeau generalized constants $\GenK$, showing, amongst other things, that the maximal ideals in $\GenK$ are exactly the topological closures of the prime ideals in $\GenK$. Subsequently, the minimal prime ideals were characterized \cite{AJOS2006}. The main focus of investigations on ideals in $\GenK$ has however been on the maximal ideals. 
Many of the properties of maximal and prime ideals can be seen in a more general context. E.g., the characterization of the maximal ideals is an almost direct consequence of the characterization of the (topologically) closed ideals (theorems \ref{thm_closed_characterization} and \ref{thm_max_ideals}); prime ideals can be characterized as those ideals that are both irreducible and radical (theorem \ref{thm_prime_characterization}). The radical ideals in turn are exactly the idempotent ideals (proposition \ref{prop_radical}), and the irreducible ideals are exactly the pseudoprime ideals (theorem \ref{thm_pseudoprime}). The pure ideals are exactly the ideals generated by idempotents (see the next section for the definitions of the used terms). The minimal prime ideals are exactly the pure prime ideals (prop.\ \ref{prop_minimal_primes}). The projective ideals are the ideals generated by a family of mutually orthogonal idempotents (theorem \ref{thm_projective_ideals}). Some of these characterizations follow easily from the theories of exchange rings and lattice-ordered rings. We answer those questions in \cite{AJ2001} about prime ideals in $\GenK$ that were still open: the Krull dimension of $\GenK$ is infinite; however, if one assumes the continuum hypothesis, there are minimal primes that are maximal. The bijective correspondence of maximal ideals $M$ with certain ultrafilters on $(0,1)$, which is implicitly already in \cite{AJ2001}, together with the characterization of the closed ideals, yields a canonical isomorphism between quotients $\GenK/M$ and nonstandard fields of asymptotic numbers ${}^\rho \K$ \cite{Lightstone, Todorov99, Todorov2004} (theorem \ref{thm_quotient_with_max_is_nonstandard_field}). In analogy with a $z$-ideal in the ring $\Cnt{}(X)$ of continuous functions on a topological space $X$, G.~Mason \cite{Mason73} introduced a notion of $z$-ideal in an arbitrary commutative ring with $1$. The $z$-ideals in $\GenK$ have a particularly concrete characterization replacing the zeroes (points of $X$) of the elements of $\Cnt{}(X)$ by subsets of $(0,1)$.
%It appears that that the intersection of a maximal ideal in the algebra of generalized functions $\Gen(\Omega)$ with $\GenK$, although in general not (topologically) closed, is always a $z$-ideal.
Finally, we have a look at ideals from the point of view of the theory of topological $\GenK$-modules \cite{Garetto2005}:
%in contrast with the situation in classical Hilbert spaces, for an ideal (=submodule) $M$ of $\GenK$, not necessarily $M^{\perp\perp}=\overline M$, and,
in contrast with the situation in classical Banach spaces, the Hahn-Banach extension property does not hold in $\GenK$, viewed as a module over itself (theorem \ref{thm_Hahn-Banach}); hence the Hahn-Banach extension property also does not hold in any $\GenK$-module which contains $\GenK$ as a topological submodule. This answers a long-standing open question (e.g.\ raised in \cite{Mayerhofer2006}) in the negative.

\section{Preliminaries}
\subsection{The rings $\GenR$ and $\GenC$}
We denote by $\K$ the field $\R$ or $\C$ and by $\GenK$ the ring $\GenR$ or $\GenC$ of real, resp.\ complex, Colombeau generalized numbers. Recall that \cite[1.2.31]{GKOS} by definition $\GenK=\Mod(\K)/\Null(\K)$, where
\begin{align*}
\Mod(\K)&=\{(x_\eps)_\eps\in \K^{(0,1)}: (\exists a\in\R)(\exists\eps_0\in(0,1))(\forall\eps\le\eps_0)(\abs{x_\eps}\le\eps^a)\}\\
\Null(\K)&=\{(x_\eps)_\eps\in \K^{(0,1)}: (\forall a\in\R)(\exists\eps_0\in(0,1))(\forall\eps\le\eps_0)(\abs{x_\eps}\le\eps^a)\}.
\end{align*}
The ring $\GenK$ arises naturally as the ring of constants of the Colombeau (differential) algebras $\Gen(\Omega)$, where $\Omega$ is a non-empty open subset of $\R^d$ \cite[1.2.35]{GKOS}.\\
We denote by $\alpha\in\GenR$ the generalized number with representative $(\eps)_\eps$.\\
For $S\subseteq(0,1)$, we denote by $e_S\in\GenR$ the generalized number with as representative the characteristic function $(\kar{S}(\eps))_\eps$, and $\co S=(0,1)\setminus S$. Clearly, $e_S\ne 0$ iff $0\in \overline S$ and $e_S\ne 1$ iff $0\in\overline{\co S}$. For a subset $A$ of a topological space, we denote by $\overline A$ the topological closure of $A$.\\
As in~\cite{AJ2001}, $\mathcal S=\{S\subseteq (0,1): 0\in \overline S \cap \overline{\co S}\}$ and $P_*(\mathcal S)$ is the set of all $\mathcal F\subseteq\mathcal S$ which are closed under finite union and such that for each $S\in\mathcal S$, either $S$ or $\co S$ belongs to $\mathcal F$. For $\mathcal F\subseteq\mathcal S$, the ideal generated by $\{e_S: S\in\mathcal F\}$ is denoted by $g({\mathcal F})$.\\
$\GenK$ is a reduced (or semiprime) ring, i.e., a ring without (nonzero) nilpotent elements. Elements in $\GenK$ are either invertible, either zero divisors \cite[1.2.39]{GKOS}. $\GenK$ is neither noetherian, neither artinian \cite[Thm.~4.5]{AJ2001}. The Jacobson radical of $\GenK$ vanishes \cite[Thm.~4.12]{AJ2001}. $\GenK$ is a ring with uncountably many maximal ideals \cite[Thm.~4.28]{AJ2001}. $\GenK$ is not Von Neumann regular \cite{AJOS2006}. The set of idempotent elements of $\GenK$ (i.e., $\{e\in\GenK: e^2=e\}$) equals $\{e_S: S\subseteq (0,1)\}$ \cite{AJOS2006}.\\
$\GenK$ is a complete topological ring with the so-called sharp topology \cite{AJ2001,Biagioni1990,Scarpalezos93}, which can be defined as follows. Let $x\in\GenK$ and $(x_\eps)_\eps$ a representative of $x$. Let
\[\val{x}=\sup\{a\in\R: (\exists\eps_0\in (0,1))(\forall \eps\le\eps_0)(\abs{x_\eps}\le\eps^a)\}.\]
Then $\sharpnorm{x}=e^{-\val{x}}$ defines an ultrapseudonorm \cite{Garetto2005} on $\GenK$ which induces a topology through the ultrametric $d(x,y)=\sharpnorm{x-y}$.

\subsection{Exchange rings and Gelfand rings}
The direct sum (of rings or modules) is denoted by $\oplus$.\\
A commutative ring $R$ with 1 is an exchange ring (or a clean ring, or a topologically boolean ring) \cite[Thm.~1.7]{McGovern2006} iff one of the following equivalent conditions is satisfied:
\begin{enumerate}
\item for each $a\in R$, there exists an idempotent $e\in R$ such that $a + e$ is invertible
\item for each $a\in R$, there exists an idempotent $e\in R$ such that $a-e$ is invertible
\item for each ideal $I$ of $R$ and $a\in R$ with $a-a^2\in I$, there exists and idempotent $e\in R$ such that $e-a\in I$ (idempotents can be lifted modulo every ideal)
\item for each $M\ne N$ maximal ideals of $R$, there exists an idempotent $e\in M\setminus N$
\item for each $a,b\in R$ with $a + b = 1$, there exists an idempotent $e\in R$ such that $e\in aR$ and $1-e\in bR$
%the following is Contessa's characterization: the previous characterization is in [ChenW, "A question on strongly clean rings", Comm.Alg.\ 34: 2347–2350, 2006] and would already be in [Goodearl, Warfield 1976, Math.Ann.\ 233:157–168] and in [Nicholson 1977, Trans.A.M.S.\ 229:269–278] 
%\item for each $a,b\in R$ with $a + b = 1$, there exist $r$, $s$ $\in R$ and an idempotent $e\in R$ such that $(1+ar)(1+bs)=0$, $1+ar\in eR$ and $1 + bs \in (1-e)R$
\item let $M$, $N$, $P$, $Q_i$ be $R$-modules such that $M=N\oplus P=\bigoplus_{i=1}^m Q_i$ with $N\cong R$ ($m\in\N$). Then there exist submodules $Q_i'\subseteq Q_i$ ($i=1$,\dots, $m$) such that $M= N\oplus \bigoplus_{i=1}^m Q_i'$ ($R$ has the finite exchange property).
\end{enumerate}
\begin{rem}
For arbitrary rings, the previous conditions are not necessarily equivalent. The class of commutative exchange rings with 1 is closed under direct products and homomorphic images \cite{McGovern2006}.
\end{rem}
%If $2$ is invertible in $R$, then $R$ is an exchange ring iff for each $a\in R$, there exists an invertible $u\in R$ and $b\in R$ with $b^2=1$ such that $a=u+b$ [CamilloV-ea-1994].

\begin{prop}\label{prop_GenK_is_exchange}
$\GenK$ is an exchange ring.
\end{prop}
\begin{proof}
We show that property 1 holds. Let $a\in \GenK$ with representative $(a_\eps)_\eps$. Let $T=\{\eps\in(0,1): \abs{a_\eps}\le 1/2\}$. Then $e_T^2=e_T$, $e_T\abs{a}\le e_T/2$ and $e_{\co T}\abs{a}\ge e_{\co T}/2$. Hence
\[\abs{e_T + a}= e_T\abs{1 + a} + e_{\co T}\abs{a}\ge e_T-e_T\abs{a} + e_{\co T}\abs{a}\ge \frac{e_T + e_{\co T}}{2}=\frac{1}{2},\]
so $e_T + a$ is invertible \cite[Thm.\ 1.2.38]{GKOS}.
\end{proof}

A commutative ring $R$ with 1 is a Gelfand ring \cite[Prop.~1.3]{McGovern2006} iff one of the following equivalent conditions is satisfied:
\begin{enumerate}
\item for each $M\ne N$ maximal ideals of $R$, there exist $a\in R\setminus M$, $b\in R\setminus N$ such that $ab=0$.
\item for each $a,b\in R$ with $a+b=1$, there exist $r,s\in R$ such that $(1+ar)(1+bs)=0$
\item every prime ideal is contained in a unique maximal ideal.
\end{enumerate}
As every commutative exchange ring with 1 is a Gelfand ring \cite[Thm.~1.7]{McGovern2006}, $\GenK$ is a Gelfand ring.

\subsection{$l$-rings and $f$-rings}
$\GenR$ is a partially ordered ring for the order $\le$, where $a\le b$ iff there exist representatives $(a_\eps)_\eps$ of $a$ and $(b_\eps)_\eps$ of $b$ such that $a_\eps\le b_\eps$, $\forall\eps\in(0,1)$.
This turns $\GenR$ into an $l$-ring (or lattice-ordered ring): for $a$, $b$ $\in\GenR$, the supremum $a\vee b$ is given on representatives as $(\max(a_\eps,b_\eps))_\eps$, the infimum $a\wedge b$ is given on representatives as $(\min(a_\eps,b_\eps))_\eps$.\\
A commutative ring $R$ with 1 is an $f$-ring if $R$ is an $l$-ring satisfying one of the following equivalent conditions \cite{Banaschewski2004, BKW}:
\begin{enumerate}
\item $R$ is isomorphic with a subdirect product of totally ordered rings, i.e., $R$ is isomorphic (as an ordered ring) with a subring $S$ of a direct product $\prod_\lambda R_\lambda$ of totally ordered rings $R_\lambda$, and each projection $S\to R_\lambda$ is surjective
\item for each $a$, $b$, $c$ $\in R$ with $a\ge 0$, $b\ge 0$, $c\ge 0$, $a\wedge b=0$ implies that $a\wedge b c = 0$
\item for each $a$, $b$, $c$ $\in R$ with $c\ge 0$, $(a\wedge b)c=ac\wedge bc$.
\end{enumerate}
%The class of commutative $f$-rings with $1$ is closed under direct products and homomorphic images \cite{BKW}.
\begin{prop}
$\GenR$ is an $f$-ring.
\end{prop}
\begin{proof}
Let $a$, $b$, $c$ $\in \GenR$ with $c\ge 0$. Fix representatives $(a_\eps)_\eps$ of $a$, $(b_\eps)_\eps$ of $b$ and $(c_\eps)_\eps$ of $c$ with $c_\eps\ge 0$, $\forall\eps$. Then $\min(a_\eps, b_\eps)c_\eps =\min(a_\eps c_\eps, b_\eps c_\eps)$, $\forall\eps$, hence $(a\wedge b)c=ac\wedge bc$.
%Let $a$, $b$, $c$ $\in \GenR$ with $a\ge 0$, $b\ge 0$, $c\ge 0$. There exists $N\in\N$ such that $c\le\alpha^{-N}$. If $a\wedge b=0$, then $0\le a\wedge (bc)\le (\alpha^{-N} a)\wedge(\alpha^{-N} b)=\alpha^{-N} (a\wedge b)= 0$, so $a\wedge bc =0$.
\end{proof}
%Alternatively, one can obtain this result from the following proposition: if $R$ is a reduced $l$-ring with $1$ and $1 \wedge x \ne 0$, for each $x\in R\setminus \{0\}$, then $R$ is an $f$-ring \cite[Prop.~12.3.20]{BKW}.

For $a\in R$, we denote $\Ann(a)=\{x\in R: xa=0\}$.

A commutative $f$-ring $R$ with $1$ is called normal if it satisfies one of the following equivalent conditions \cite{Larson88}:
\begin{enumerate}
\item for each $a\in R$, $R=\Ann(a\vee 0)+\Ann(a\wedge 0)$
\item for each $a,b \in R$ with $a\wedge b =0$, $R=\Ann(a)+\Ann(b)$
%\item for each proper prime $l$-ideal $P$, $\{x\in R: (\exists y\notin P)(xy=0)\}$ is prime.%skipped this equivalent condition because $l$-ideal is only defined later
\end{enumerate}
A reduced commutative $f$-ring with $1$ is normal iff \cite{Huijsmans74}
\begin{enumerate}
\item for each $a,b \in R$ with $a b =0$, $R=\Ann(a)+\Ann(b)$
\item for each $P_1\ne P_2$ minimal prime ideals of $R$, $R=P_1+P_2$
\item each proper prime ideal $P$ of $R$ contains a unique minimal prime ideal
\item for each $a,b\in R$, $\Ann(ab)=\Ann(a)+\Ann(b)$.
\end{enumerate}

\begin{lemma}\label{lemma_zero_divisors}
Let $x$, $y$ $\in\GenK$. The following are equivalent:
\begin{enumerate}
\item $xy=0$
\item there exists $S\subseteq(0,1)$ such that $xe_S= 0$ and $ye_{\co S}=0$
\item $\Ann(x)+\Ann(y)=\GenK$
\item $\abs{x}\wedge\abs{y}=0$.
\end{enumerate}
\end{lemma}
\begin{proof}
$(1)\implies (2)$: fix representatives $(x_\eps)_\eps$ of $x$, $(y_\eps)_\eps$ of $y$ and let $S=\{\eps\in (0,1): \abs{x_\eps}\le\abs{y_\eps}\}$. Then $0\le (\abs{x}e_S)^2\le\abs{x}\abs{y}e_S=0$, so $xe_S=0$; similarly $ye_{\co S}=0$.\\
$(2)\implies (3)$: $e_S\GenK\subseteq \Ann(x)$ and $e_{\co S}\GenK\subseteq \Ann(y)$, so $\GenK=e_S\GenK + e_{\co S}\GenK \subseteq \Ann(x) + \Ann(y)$.\\
$(3)\implies (1)$: if $1= a + b$, for some $a, b\in \GenK$ with $ax=by=0$, then $xy=xy(a + b)=0$.\\
$(1)\Leftrightarrow (4)$: as $xy=0$ iff $\abs{x}\abs{y}=0$, we may suppose $x$, $y$ $\in\GenR$. The equivalence holds in any reduced $f$-ring \cite[Thm.~9.3.1]{BKW}, hence also in $\GenR$.
\end{proof}
\begin{cor}
$\GenR$ is a (reduced) normal $f$-ring.
\end{cor}

\subsection{Ideals}
Let $R$ be a commutative ring with $1$. We write $I\idealproper R$ iff $I$ is a proper ideal of $R$ (i.e., an ideal different from $R$ itself). A not necessarily proper ideal is denoted by $I\ideal R$.\\
%Recall that for $I$, $J$ $\ideal$ $\GenK$, $IJ$ is defined as the ideal generated by products $ab$, where $a\in I$, $b\in J$, i.e., \[IJ=\big\{\sum_{i=1}^n a_i b_i : n\in\N, a_i\in I, b_i\in J\big\}.\]
We adopt the convention that $R$ is not a prime ideal of $R$.\\
$I\ideal R$ is called projective iff $I$ is projective as an $R$-module \cite{Cartan}.\\
$I\ideal R$ is called idempotent iff $I^2=I$.\\
We denote the radical of $I\ideal R$ by $\rad{I}=\{x\in R: (\exists n\in\N)(x^n\in I)\}=\bigcap_{I\subseteq P\atop P \text{ prime }} P$ (e.g., see \cite[0.18]{GJ}).\\
%It is easily seen that $I\subseteq\rad{I}\ideal\GenK$. Further, $1\in I$ iff $1\in\rad{I}$, so if $I\idealproper\GenK$, then also $\rad{I}\idealproper\GenK$.\\
$I\ideal R$ is called radical (or semiprime) iff $I=\rad I$, or equivalently, iff $(\forall x\in R)(x^2\in I\implies x\in I)$.\\
We denote the annihilator ideal of $I\ideal R$ by $\Ann(I)=\{x\in R:(\forall a\in I)(xa=0)\}$.\\
$I\ideal R$ is called pseudoprime iff for each $a,b\in R$, $ab=0$ implies $a\in I$ or $b\in I$.\\
$I\ideal R$ is called irreducible (or meet-irreducible) iff for each $J, K\ideal \GenK$, $I=J\cap K$ implies $I=J$ or $I=K$ \cite[\S 6]{Matsumura}.\\
The ideal generated by $x_\lambda$, $\lambda\in \Lambda$ (for some index-set $\Lambda$) is denoted by $\lspan{x_\lambda: \lambda\in \Lambda}$.
\begin{prop}(e.g.~\cite[0.16]{GJ})\label{prop_prime_constructor}
Let $A\subset R$ closed under multiplication. Let $I$ be a proper ideal of $R$ with $I\cap A =\emptyset$. Then there exists an ideal $P\supseteq I$ maximal w.r.t.\ the property that $P\cap A =\emptyset$. $P$ is a prime ideal.
\end{prop}

An ideal $I\ideal R$ is pure if it satisfies one of the following equivalent conditions \cite[Prop.~7.2]{Borceux83}:
\begin{enumerate}
\item $(\forall x\in I)(\exists y\in I)(x=xy)$
\item $(\forall J\ideal R)(IJ=I\cap J)$
\item $(\forall x\in I)(I + \Ann(x)= R)$.
\end{enumerate}
We denote by $m(I)$ the pure part of $I\ideal R$, i.e., the largest pure ideal contained in $I$. By definition, $I$ is pure iff $I=m(I)$. For $I$, $J$ $\ideal R$, $m(I\cap J)=m(I)\cap m(J)$ \cite[Prop.~7.9]{Borceux83}. If $R$ is a Gelfand ring, then \cite[\S 8.2--3]{Borceux83}
\[m(I)=\{x\in R: (\exists y\in I)(x=xy)\},\]
\[(\forall I,J\ideal R)(I+J=R\implies m(I) + J = R)\]
and
\[(\forall I_\lambda\ideal R, \lambda\in\Lambda)\Big(\sum_{\lambda\in\Lambda} m(I_\lambda)=m\Big(\sum_{\lambda\in\Lambda} I_\lambda\Big)\,\Big).\]
It is not hard to see that then
\[
m(I)=\bigcup_{x\in I}\Ann(1-x)=\{x\in R: I+\Ann(x) = R\}.
\]
If $R$ is an exchange ring, the pure ideals are exactly the ideals generated by idempotents \cite[Thm.~1.7]{McGovern2006}.
%\cite[Ex.~20]{Borceux83} for the easy direction of the statement
It is not hard to see that then
\[
m(I)=\{x\in R: (\exists e = e^2 \in I)(x=xe)\} =\lspan{e\in I : e^2 = e}.
\]
In particular, for $I\ideal\GenK$, $m(I)=\lspan{e_S: e_S \in I}$.

Let $R$ be an $l$-ring. An ideal $I\ideal R$ is an $l$-ideal (or absolutely order convex) iff for each $x\in I$ and $y\in R$, $\abs{y}\le\abs{x}$ implies that $y\in I$. %Notation: $I\lideal R$.
Every ideal in $\GenR$ is an $l$-ideal \cite{AJOS2006}.

\section{Correspondence between ideals in $\GenR$ and $\GenC$}
In order to transfer some results about ideals of $\GenR$ (obtained e.g., by the $l$-ring structure) to $\GenC$, we use the bijective correspondence in \cite{AJOS2006}, which we can put in a more general context.

\begin{df}
Let $A$ be a commutative, faithful algebra with $1$ over an $l$-ring $R$ (hence $R$ can be identified with a subring of $A$). We call $A$ an $R$-normed algebra if there exists a map $\norm{.}$: $A\to R$ with the following properties for each $a,b\in A$ and $r\in R$:
\begin{enumerate}
\item $\norm{a}\ge 0$ and $(\norm{a}=0\iff a=0)$
\item $\norm{ab} \le \norm{a}\norm{b}$
\item $\norm{a+b}\le\norm{a}+\norm{b}$
\item $\norm{r}=\abs{r}$.
\end{enumerate}
We call an ideal $I\ideal A$ norm convex iff for each $a\in I$ and $b\in A$, $\norm{b}\le\norm{a}$ implies that $b\in I$.
\end{df}

\begin{prop}
Let $R$ be an $l$-ring and let $A$ be a (commutative, faithful) $R$-normed $R$-algebra (with $1$). Then the maps $I\mapsto I\cap R$ (for $I$ a norm convex ideal of $A$) and $J\mapsto \lspan J=\{x\in A: \norm{x}\in J\}$ (for $J$ an $l$-ideal of $R$)
define a lattice isomorphism between the lattice of $l$-ideals of $R$ and the lattice of norm convex ideals of $A$.
\end{prop}
\begin{proof}
It is easy to see that $I\cap R$ is an $l$-ideal of $R$ if $I$ is a norm convex ideal of $A$, and that $\{x\in A:\norm{x}\in J\}$ is a norm convex ideal of $A$ if $J$ is an $l$-ideal of $R$. Moreover, $\lspan J$ is the smallest norm convex ideal of $A$ that contains $J$. Further, $\lspan{I\cap R}= \{x\in A: \norm x\in I\}= I$ for each norm convex ideal $I$ of $A$ and $\lspan{J}\cap R=\{x\in R: \abs{x}\in J\}= J$ for each $l$-ideal $J$ of $R$, so the correspondence is bijective. As both operations clearly preserve the order $\subseteq$ and the $l$-ideals of an $l$-ring form a lattice, the correspondence defines a lattice isomorphism.
\end{proof}
%I will not discuss the preservation of further operations in the general context, because then one has to ensure that these operations yield $l$-ideals, which may depend on the properties of the $l$-ring; or one has to distinguish the $l$-definition and the algebraic definitions of the operation, which is all irrelevant for our application
In particular, $\GenC$ is a $\GenR$-normed $\GenR$-algebra (with the usual absolute value on $\GenC$ as norm), and every ideal in $\GenC$ is norm convex \cite{AJOS2006}. By the previous proposition, we obtain a lattice isomorphism between the lattice of ideals of $\GenR$ and the lattice of ideals of $\GenC$. Hence arbitrary sums and intersections are preserved. One easily checks that this isomorphism also preserves products, principal ideals, prime  and pseudoprime ideals (because of $\abs{ab}=\abs{a}\abs{b}$). It follows that all operations on ideals that can be defined in terms of those operations are preserved, e.g., maximal, idempotent and pure ideals, radicals, annihilators.

\section{Prime, pseudoprime and radical ideals in $\GenK$}
\begin{lemma}\label{lemma_principal_ideals}
Let $a$, $b\in\GenK$.
\begin{enumerate}
\item $a\GenK + b\GenK=(\abs{a}\vee \abs{b})\GenK=(\abs{a} + \abs{b})\GenK$.\\In particular, $\GenK$ is a Bezout ring, i.e., every finitely generated ideal in $\GenK$ is a principal ideal.
\item $a\GenK \cap b\GenK=(\abs{a}\wedge \abs{b})\GenK$.
%\item $a\GenK \cdot b\GenK=ab\GenK$.%holds in any commutative ring, almost by definition
\end{enumerate}
\end{lemma}
\begin{proof}
(1) The corresponding statement holds in any $l$-ring in which every ideal is an $l$-ideal \cite[Prop.~8.2.8]{BKW}, in particular in $\GenR$.\\%As $\abs{a}\le \abs{a}\vee \abs{b}$, by absolute order convexity, $a\in (\abs{a}\vee \abs{b})\GenK$. Similarly, $b\in(\abs{a}\vee \abs{b})\GenK$. So $a\GenK + b\GenK\subseteq (\abs{a}\vee \abs{b})\GenK$. Further, again by absolute order convexity, $(\abs{a}\vee\abs{b})\GenK\subseteq(\abs{a}+\abs{b})\GenK\subseteq a\GenK + b\GenK$.\\
(2) The corresponding statement holds in any $f$-ring in which every ideal is an $l$-ideal \cite[Prop.~9.1.8]{BKW}, in particular in $\GenR$.\\%As $\abs{a}\wedge\abs{b}\le\abs{a}$, by absolute order convexity, $\abs{a}\wedge\abs{b}\in a\GenK$. Similarly, $\abs{a}\wedge\abs{b}\in b\GenK$. So $(\abs{a}\wedge\abs{b})\GenK \subseteq a\GenK\cap b\GenK$. Conversely, let $x\in a\GenK\cap b\GenK$. So $x=\abs{a}y=\abs{b}z$, for some $y$, $z\in \GenK$. Fix representatives $(a_\eps)_\eps$ of $a$ and $(b_\eps)_\eps$ of $b$ and let $S=\{\eps\in (0,1): \abs{a_\eps}\ge \abs{b_\eps}\}$. Then $x=xe_{\co S}+xe_{S}=\abs{a}ye_{\co S}+\abs{b}ze_{S}=(ye_{\co S} + z e_S)(\abs{a} e_{\co S}+ \abs{b} e_S)=(ye_{\co S} + z e_S)(\abs{a}\wedge\abs{b})$. So $a\GenK\cap b\GenK\subseteq(\abs{a}\wedge\abs{b})\GenK$.
The bijective correspondence of ideals yields the result for $\GenC$.
\end{proof}

\begin{lemma}\label{lemma_ideal_calc}
Let $I\ideal\GenK$ and $m\in\N$.
\begin{enumerate}
\item $I^m=\{x\in\GenK: \sqrt[m]{\abs x}\in I\}$%=\{u a^m : a\in I, \abs{u}=1\}$.
\item Let $J\ideal\GenK$, $J^m\subseteq I^m$. Then $J\subseteq I$.\\
In particular, if $x\in\GenK$ and $x^m\in I^m$, then $x\in I$.
\item $\rad I=\lspan[\big]{\sqrt[n]{\abs{x}}: n\in\N, x\in I}$. In particular, for $x\in\GenK$, $\rad{x\GenK}=\lspan[\big]{\sqrt[n]{\abs x}: n\in\N}$.
\end{enumerate}
\end{lemma}
\begin{proof}
(1) As $\GenR$ is an $l$-ring in which every ideal is an $l$-ideal, $I^m=\{a\in \GenR: (\exists x\in I)(\abs{a}\le\abs{x}^n)\}$ \cite[Prop.~8.2.11]{BKW}. By the bijective correspondence of ideals, this also holds in $\GenC$. Taking $n$-th roots and using the fact that $I$ is an $l$-ideal, the result follows.\\
%As in any commutative ring, $I^m=\{\sum_{j=1}^n a_{1,j}\cdots a_{m,j} : n\in\N, a_{i,j}\in I\}$. Now let $a_1$, \dots, $a_m\in I$ and let $c:=\abs{a_1}\vee\cdots\vee\abs{a_m}$. Then $c\in I$ by lemma~\ref{lemma_principal_ideals} and $\abs{a_1\cdots a_m}\le c^m$. By absolute order convexity of ideals, $a_1\cdots a_m=xc^m$, for some $x\in \GenK$. As $x=u\abs{x}=u\sqrt[m]{\abs{x}}^m$, for some $u\in\GenK$ with $\abs{u}=1$, this means that $a_1\cdots a_m= u (\sqrt[m]{\abs{x}}c)^m$. So $I^m=\{\sum_{j=1}^n u_j a_j^m : n\in\N, a_j\in I, \abs{u_j}=1\}$. Now let $a$, $b\in I$ and $u$, $v\in\GenK$ with $\abs{u}=\abs{v}=1$. Then $d:=\abs{ua^m + vb^m}^{\frac{1}{m}}\le(\abs{a}^m+\abs{b}^m)^{\frac{1}{m}}\le \sqrt[m]{2}(\abs{a}\vee\abs{b})$. By lemma~\ref{lemma_principal_ideals}, $\abs{a}\vee\abs{b}\in I$, so by absolute order convexity of $I$, $d\in I$ and $ua^m + v b^m = w\abs{u a^m + v b^m}=w d^m$, for some $w\in\GenK$, $\abs{w}=1$. So $I^m=\{u a^m : a\in I, \abs{u}=1\}$.\\
(2) If $x\in J$, then $x^m\in J^m\subseteq I^m$, so by part~1, $\abs{x}=\sqrt[m]{\abs{x^m}}\in I$, hence $x\in I$ as $I$ is an $l$-ideal.\\
(3) $\subseteq$: if $a^n=x\in I$, for some $n\in\N$, then $a\in \abs{a}\GenK=\sqrt[n]{\abs{x}}\GenK$.\\
$\supseteq$: if $x\in I$, then $(\sqrt[n]{\abs{x}})^n\in I$.
\end{proof}

%obsolete
%\begin{lemma}\label{lemma_ideal_calc}
%Let $I\ideal\GenK$ and $x\in\GenK$.
%If $x^2\in I^2$, then $x\in I$.
%\end{lemma}
%\begin{proof}
%As $x^2\in I^2$, $x^2=\sum_{i=1}^n a_i b_i$, for some $a_i, b_i\in I$. We proceed by induction on $n$.\\
%So let first $x^2=ab$, $a,b\in I$. Choose representatives $(a_\eps)_\eps$ of $a$ and $(b_\eps)_\eps$ of $b$. Let $S=\{\eps\in(0,1): \abs{a_\eps}\ge\abs{b_\eps}\}$. Then $0\le\abs{x}^2e_S=\abs{ab}e_S\le\abs{a}^2 e_S$, so $\abs{x}e_S\le\abs{a}e_S$. By absolute order convexity of $I$, $xe_S\in I$. Similarly, $xe_{\co S}\in I$. So $x=xe_S+xe_{\co S}\in I$.\\
%Now let $x^2=\sum_{i=1}^n a_i b_i$, with $a_i, b_i\in I$. Let \[T=\{\eps\in(0,1): \abs{a_nb_n}\ge\abs{a_1b_1+\cdots+a_{n-1}b_{n-1}}\}.\] Then $\abs{x}^2e_T\le \abs{a_1b_1+\cdots+a_{n-1}b_{n-1}}e_T+\abs{a_nb_n}e_T\le 2\abs{a_nb_n}e_T$. By absolute convexity of the ideal $a_nb_n\GenK$, $(xe_T)^2= ab$ for some $a$, $b$ $\in I$. By the case $n=1$, $xe_T\in I$. Similarly, $\abs{x}^2e_{\co T}\le 2\abs{a_1b_1+\cdots+a_{n-1}b_{n-1}}e_T$, and $(xe_{\co T})^2=\sum_{i=1}^{n-1} \tilde a_i \tilde b_i$, for some $\tilde a_i, \tilde b_i\in I$. By the induction hypothesis, $x e_{\co T}\in I$. So $x=xe_T+xe_{\co T}\in I$.
%\end{proof}

\begin{prop}\label{prop_radical}
The following are equivalent for an ideal $I\ideal\GenK$:
\begin{enumerate}
\item $I$ is idempotent
\item $I$ is radical
%\item $I=\rad{J}$, for some $J\ideal\GenK$
%\item $(\forall x\in \GenK)(\forall n\in\N)(x^n\in I\implies x\in I)$
%\item $(\forall x\in \GenK)(x^2\in I\implies x\in I)$ ($I$ is semiprime)
\item $(\forall x\in I)(\sqrt{\abs x}\in I)$
%\item $(\forall x\in I, x\ge 0)(\sqrt{x}\in I)$
\item $I$ is an intersection of prime ideals.
\end{enumerate}
%If $I\idealproper\GenK$ is prime, then $I$ is idempotent.
\end{prop}
\begin{proof}
$1\implies 2$: let $x\in\rad{I}$. So $x^n\in I=I^n$, for some $n\in\N$. By lemma~\ref{lemma_ideal_calc}, $x\in I$.\\
%$2\implies 3$: a fortiori.\\
%$3\implies 4$: if $x^n\in I$, then $x^{nm}\in J$, for some $m\in\N$. So $x\in\rad{J}=I$.\\
%$4\implies 5$: a fortiori.\\
$2\implies 3$: as $I$ is an $l$-ideal, if $x\in I$, $\abs x= (\sqrt{\abs x})^2\in I$.\\
%$6\implies 7$: a fortiori.\\
%$7\implies 2$: let $x\in\rad{I}$, i.e., $x^n\in I$ for some $n\in\N$. So $\abs{x}^n\in I$. As $x\in\GenK$ is moderate, $\abs x\le\alpha^{-N}$, for some $N\in\N$. So for $y=\alpha^N \abs{x}$, $y^n=\alpha^{nN} \abs{x}^n\in I$ and $0\le y\le 1$. If $n>1$, then $0\le y^{n-1}\le y^{n/2}=\sqrt{y^n}$, so by condition 6, $y^{n-1}\in I$. Inductively, $y\in I$. So also $\abs{x}=\alpha^{-N}y\in I$, and $x\in I$. (We have used the absolute order convexity of $I$ several times.)\\
$3\implies 1$: by lemma~\ref{lemma_ideal_calc}, $I\subseteq I^2$. The converse inclusion holds for any ideal.\\
%$8\implies 5$: if $I=\bigcap_i P_i$, $P_i\idealproper\GenK$ prime, and $x^2\in I$, then $x^2\in P_i$, $\forall i$, so $x\in P_i$, $\forall i$. So $x\in I$.\\
$2\Leftrightarrow 4$: since $\rad I=\bigcap_{I\subseteq P\atop P \text{ prime}} P$.
%Finally, if $I$ is prime, then $I$ clearly satisfies condition $5$.
\end{proof}

\begin{prop}\label{prop_radical_intersection}
\leavevmode
\begin{enumerate}
%\item The intersection of a family of radical ideals of $\GenK$ is radical.
\item For a family $(I_\lambda)_{\lambda\in\Lambda}$ of ideals $I_\lambda$ $\ideal$ $\GenK$, $\rad{\sum_{\lambda\in\Lambda }I_\lambda}=\sum_{\lambda\in\Lambda}\rad{I_\lambda}$. In particular, the sum of a family of radical ideals is radical.
\item For $I,J\idealproper\GenK$, $\rad I\cap \rad J =\rad{I\cap J}$.
\item For each $I\ideal \GenK$, \[\radpart{I}:=\bigcap_{n\in\N}I^n=\{x\in\GenK: (\forall n\in\N)(\sqrt[n]{\abs x}\in I)\}=\{x\in\GenK: \rad{x\GenK}\subseteq I\}\]
is the largest radical ideal contained in $I$. In particular, $I$ is radical iff $I=\radpart{I}$.
\item For a family $(I_\lambda)_{\lambda\in\Lambda}$ of ideals $I_\lambda$ $\ideal$ $\GenK$, $\bigcap_{\lambda\in\Lambda} \radpart{I}_\lambda= \radpart{\big(\bigcap_{\lambda\in\Lambda} I_\lambda\big)}$. In particular, the intersection of a family of radical ideals is radical.
\item For $I\ideal\GenK$, $m(I)\subseteq\radpart I\subseteq I$. In particular, every pure ideal of $\GenK$ is radical.
\end{enumerate}
\end{prop}
\begin{proof}
%(1) By proposition \ref{prop_radical} (4).\\%$\rad{\bigcap I_j}\subseteq \bigcap \rad{I_j}$.\\
(1) In any commutative ring, the sum of a family of idempotent ideals is idempotent. So by proposition \ref{prop_radical}, $\sum_{\lambda\in\Lambda}\rad{I_\lambda}$ is idempotent, hence $\sum_{\lambda\in\Lambda}\rad{I_\lambda}= \rad{\sum_{\lambda\in\Lambda} \rad{I_\lambda}} \supseteq\rad{\sum_{\lambda\in\Lambda} I_\lambda}$. The converse inclusion also holds.\\
(2) Elementary.\\
(3) The equalities follow by lemma \ref{lemma_ideal_calc}.
%$\subseteq$: Let $x\in I^n$, $\forall n\in\N$ and let $a\in \rad{x\GenK}$. So $a^n\in x\GenK\subseteq I^n$, for some $n\in\N$. By lemma \ref{lemma_ideal_calc}, $x\in I$. So $\rad{a\GenK}\subseteq I$.\\
%$\supseteq$: If $\rad{a\GenK}\subseteq I$, then $\sqrt[n]{\abs a}\in\rad{a\GenK}\subseteq I$, hence $a\in I^n$, for each $n\in\N$.\\
Let $x\in \rad{\radpart I}$. Then $\rad{x^n\GenK}\subseteq I$, for some $n\in\N$. As $\rad{J^n}=\rad{J}$, $\forall J\ideal \GenK$, $x\in \radpart I$. So $\radpart I$ is a radical ideal contained in $I$. Now let $J$ be a radical ideal contained in $I$. Then for each $x\in J$, $\rad{x\GenK}\subseteq \rad{J}=J\subseteq I$. Hence $J\subseteq \radpart I$.\\
(4) $x\in \bigcap_{\lambda\in\Lambda} \radpart{I}_\lambda$ iff $(\forall\lambda\in\Lambda)(\rad{x\GenK}\subseteq I_\lambda)$ iff $x\in \radpart{\big(\bigcap_{\lambda\in\Lambda} I_\lambda\big)}$.\\
(5) If $x\in m(I)$, then $x=xe_S$, for some $e_S\in I$, and $x=x e_S^n\in I^n$, $\forall n\in\N$. So $x\in \radpart I$.
\end{proof}
\begin{rem}
For a family $(I_\lambda)_{\lambda\in\Lambda}$ of ideals $I_\lambda$ $\ideal$ $\GenK$, $\rad{\bigcap I_\lambda}\subseteq \bigcap \rad{I_\lambda}$. Equality does not hold in general, as can be seen by the example $\radpart I =\rad{\radpart I}=\rad{\bigcap_{n\in\N}I^n}\subseteq \bigcap_{n\in\N}\rad{I^n}=\rad I$. Similarly, $\sum \radpart I_\lambda\subseteq \radpart{(\sum I_\lambda)}$. But it is easy to see that for each $n\in\N$, $\radpart{\lspan[\big]{\sqrt[n]{\abs x}: x\in I}}=\radpart I$, so equality does not hold in general: $\radpart I =\sum_{n\in\N} \radpart{\lspan[\big]{\sqrt[n]{\abs x}: x\in I}}\subseteq \radpart{\big(\sum_{n\in\N} \lspan[\big]{\sqrt[n]{\abs x}: x\in I}\big)}=\radpart{(\rad I)}=\rad I$.
\end{rem}

\begin{thm}\label{thm_pseudoprime}
The following are equivalent for an ideal $I\idealproper\GenK$:
\begin{enumerate}
%\item $I$ contains a prime ideal
\item $I$ is pseudoprime
\item the set of ideals containing $I$ is totally ordered (for $\subseteq$)
\item $I$ is irreducible
\item $(\forall S\subseteq(0,1))(e_S\in I$ or $e_{\co S}\in I)$
\item there exists ${\mathcal F}\in P_*({\mathcal S})$ such that $m(I)=g(\mathcal F)$
\item $\rad I$ is prime.
\end{enumerate}
Moreover, for $\GenK=\GenR$, these conditions are equivalent with
\begin{itemize}
\item[7.] $\GenR/I$ is totally ordered.
\end{itemize}
\end{thm}
\begin{proof}
$(1)\implies (7)$: the corresponding statement holds for any $l$-ideal in any commutative $l$-ring $R$ in which $a^2=\abs{a}^2$, $\forall a\in R$ \cite[3.5]{GK60}, hence also for any ideal in $\GenR$.\\
$(7)\implies (2)$: (cf.\ \cite{GK60}) the map $J\mapsto J/I$ is an order preserving bijection between the ($l$-)ideals of $\GenR$ containing $I$ and the $l$-ideals of $\GenR/I$. As in any totally ordered ring, the $l$-ideals in $\GenR/I$ are totally ordered.\\
$(1)\implies (2)$: the bijective correspondence of ideals yields the result for $\GenC$.\\
$(2)\implies (3)$: a fortiori.\\
$(3)\implies (4)$: as in any commutative $l$-ring with $1$ in which every ideal is an $l$-ideal, the irreducibility of $I$ is equivalent with: for any $a,b\in\GenR$, $a\GenR\cap b\GenR\subseteq I$ implies $a\in I$ or $b\in I$ \cite[Prop.~8.4.1]{BKW}. In particular, $e_S \GenR\cap e_{\co S}\GenR=\{0\}\subseteq I$, so $e_S\in I$ or $e_{\co S}\in I$. The bijective correspondence of ideals yields the result for $\GenC$ (since any $e_S\in\GenR$, property (4) is preserved by the correspondence).\\
$(4)\implies (1)$: let $a,b\in\GenK$ with $ab=0$. By lemma \ref{lemma_zero_divisors}, there exists $S\subseteq(0,1)$ such that $ae_S=0$ and $b e_{\co S}=0$. Either $e_S\in I$, hence $b=b e_S\in I$, or $e_{\co S}\in I$, hence $a=a e_{\co S}\in I$.\\
$(4)\implies (5)$: let $\mathcal F=\{S\in\mathcal S: e_S\in I\}$. As $e_S$, $e_T\in I$ imply that $e_{S\cup T}=e_S + e_T - e_S e_T\in I$, $\mathcal F\in P_*(\mathcal S)$. As $I\ne\GenK$, $g(\mathcal F)=m(I)$.\\
$(5)\implies (4)$: Let $S\subseteq (0,1)$. If $e_S=0$ or $e_S=1$, then (4) is trivially fulfilled, so we may suppose $S\in\mathcal S$. Hence either $e_S$ or $e_{\co S}$ belong to $g(\mathcal F)=m(I)\subseteq I$.\\
$(2)\implies (6)$: the intersection of a chain of prime ideals is prime, hence $\rad{I}=\bigcap_{I\subseteq P\atop P \text{ prime}} P$ is prime.\\
$(6)\implies (4)$: for any $S\subseteq (0,1)$, $e_S e_{\co S}=0\in \rad{I}$, so $e_S\in\rad{I}$ or $e_{\co S}\in \rad{I}$. Since $e_S$, $e_{\co S}$ are idempotent elements, $e_S\in I$ or $e_{\co S}\in I$.\\
(Alternatively, $(1)\Leftrightarrow (6)$ holds for any $l$-ideal in any commutative reduced normal $f$-ring with $1$ \cite[Thm.~2.6]{Larson88}.)
\end{proof}

\begin{thm}\label{thm_prime_characterization}
Let $I\idealproper\GenK$.\\
Then $I$ is prime iff $I$ is pseudoprime and radical. Or, equivalently, iff
\begin{equation}
(\forall S\subseteq(0,1))(e_S\in I \text{ or\/ } e_{\co S}\in I)\label{S condition}
\end{equation}
and
\begin{equation*}
(\forall x\in I)(\sqrt{\abs x}\in I).%\label{radical condition}
\end{equation*}
The set of (proper) prime ideals of $\GenK$ equals
\[\{\rad{I}: I\idealproper\GenK, I\supseteq g({\mathcal F}), \text{ for some }\mathcal F\in P_*(\mathcal S)\}.\]
\end{thm}
\begin{proof}
Clearly, any prime ideal is radical (prop.~\ref{prop_radical}) and pseudoprime.\\
Conversely, if $I\ideal\GenK$ is pseudoprime and radical, then $I=\rad I$ is prime by theorem \ref{thm_pseudoprime}. (Alternatively, in any commutative ring $R$ with $1$, $I\ideal R$ is prime iff $I$ is irreducible and radical \cite{Lambek73}.)
%To show sufficiency, let $xy\in I$. Consider representatives $(x_\eps)_\eps$ of $x$ and $(y_\eps)_\eps$ of $y$. Let $S=\{\eps\in(0,1): \abs{x_\eps}\le\abs{y_\eps}\}$. Then $\abs{x}e_S\le\abs{y}e_S$, so $\abs{x}^2 e_S\le \abs{x}\abs{y}e_S\le\abs{xy}$. By the absolute order convexity of $I$, $x^2 e_S\in I$. Similarly, $y^2e_{\co S}\in I$. If $e_S\in I$, then $y^2= y^2(e_S + e_{\co S})\in I$. Similarly, if $e_{\co S}\in I$, then $x^2\in I$. Now either $e_S\in I$ or $e_{\co S}\in I$, so either $x^2\in I$ or $y^2\in I$. As $I$ is idempotent, $x\in I$ or $y\in I$.
\end{proof}

\begin{prop}\label{prop_minimal_primes}\leavevmode
\begin{enumerate}
\item %For $\mathcal F\in P_*(\mathcal S)$, $g(\mathcal F)$ is prime.
The set of minimal prime ideals of $\GenK$ equals $\{g(\mathcal F): \mathcal F\in P_*(\mathcal S)\}$.
\item If $I\idealproper\GenK$ is pseudoprime, then $m(I)$ is a minimal prime ideal. In particular, a (proper) prime ideal is minimal iff it is pure.
\item An ideal of $\GenK$ is pseudoprime iff it contains a prime ideal.
\item If $I,J\idealproper\GenK$ and $J$ is pseudoprime, then $I+J\ne \GenK$ iff $m(I)\subseteq m(J)$.
\item If $I, J\idealproper\GenK$ are pseudoprime, $I + J\in \{I,J,\GenK\}$.
\item If $I\idealproper\GenK$ is pseudoprime, $\radpart I$ is the largest prime ideal contained in $I$ and $\rad I$ is the smallest prime ideal containing $I$.
\end{enumerate}
\end{prop}
\begin{proof}
(1) This is proven in \cite{AJOS2006}. Alternatively, let $\mathcal F\in P_*(\mathcal S)$. By theorem \ref{thm_prime_characterization}, it is sufficient to show that $g(\mathcal F)$ is radical. Since an ideal generated by idempotent elements is idempotent, this follows from proposition \ref{prop_radical}.\\
%So let $x\in g(\mathcal F)$. Then $x=x e_S$ for some $e_S\in g({\mathcal F})$. So $\sqrt{\abs x}=\sqrt{\abs{x}} e_S\in g(\mathcal F)$.\\
(2) Combine part~1 with theorem \ref{thm_pseudoprime}(5).\\
%By part~1, a minimal prime is generated by idempotents, hence pure. Conversely, if $I$ is a pure prime ideals, $I=m(I)=g({\mathcal F})$ with $\mathcal F=\{S\subseteq (0,1): e_S\in I\}$. As $I$ is prime, ${\mathcal F}'={\mathcal F}\setminus \{S\subseteq(0,1): e_S = 0\}\in P_*({\mathcal S})$ and $I=g({\mathcal F}')$, so it is a minimal prime by part~1.\\
%(3) If $P_1$, $P_2$ are prime ideals such that $P_1\not\subseteq P_2$ and $P_2\not\subseteq P_1$, then $P_1\cap P_2$ is not prime (see e.g.\ \cite[2.10]{GJ}). Now if $P_1$, $P_2$ are prime ideals containing $g(\mathcal F)$, $P_1\cap P_2$ is idempotent by lemma~\ref{radical intersection}, and it contains $g(\mathcal F)$. So by theorem~\ref{thm_prime_characterization}, $P_1\cap P_2$ is prime. So $P_1\subseteq P_2$ or $P_2\subseteq P_1$.\\
(3) By part~2, a pseudoprime ideal contains a prime ideal. The converse implication holds in any commutative ring with $1$.\\
(4) If $m(I)\not\subseteq m(J)$, then there exists an idempotent $e_S\in m(I)\setminus m(J)$, since $m(I)$ is generated by idempotents. By part~2, $m(J)$ is prime, so $e_{\co S}\in m(J)$ and $1=e_S + e_{\co S}\in m(I) + m(J)\subseteq I+J$. Hence $I+J=\GenK$. Conversely, if $m(I)\subseteq m(J)$, then $m(I+J)= m(I)+ m(J)= m(J)\subsetneqq \GenK=m(\GenK)$. Hence $I+J\ne\GenK$.\\
(5) If $m(I)= m(J)$, then $I$, $J$ are ideals that contain a common prime ideal, hence $I\subseteq J$ or $J\subseteq I$ by theorem \ref{thm_pseudoprime}, so $I+J\in\{I,J\}$. Otherwise, $I+J=\GenK$ by part~4.\\%(Alternatively, one can use the fact that in any normal reduced commutative $f$-ring $R$ with $1$, the sum of two different minimal prime ideals equals $R$.)\\
(6) $m(I)\subseteq \radpart I\subseteq \rad I$, so $\radpart I$, $\rad I$ are prime by theorem \ref{thm_prime_characterization}. As every prime ideal is radical, the statements follow by the definitions of $\radpart I$ and $\rad I$.
\end{proof}

\begin{prop}\label{prop_radical_advanced}
Let $I,J\idealproper\GenK$.
\begin{enumerate}
\item If $J$ is pseudoprime and $I\cap J$ is radical, then $I$ is radical or $J$ is prime.
\item If $I$ and $J$ are pseudoprime, $I\cap J$ is radical, $I\not\subseteq J$ and $J\not\subseteq I$, then $I$ and $J$ are prime.
\item Let $J$ be pseudoprime and $I\not\subseteq J$. If $m(I)\subseteq m(J)$, then $I+J$ is prime or $I+J=I+m(J)$. If $m(I)\not\subseteq m(J)$, then $I+J=\GenK$.
\item Let $I$ be radical, $J$ pseudoprime and $I\not\subseteq J$. If $m(I)\subseteq m(J)$, then $I+J$ is prime. If $m(I)\not\subseteq m(J)$, then $I+J=\GenK$.
\item $I=\bigcap_{I\subseteq P\atop P \text{ pseudoprime}}P$.
\end{enumerate}
\end{prop}
\begin{proof}
(1) Let $P=(I\cap J)+m(J)$. As $m(J)\subseteq P$, $P$ is pseudoprime by proposition \ref{prop_minimal_primes}. As $I\cap J$ is radical and $m(J)$ is pure (hence radical by proposition \ref{prop_radical_intersection}), $P$ is radical. By theorem \ref{thm_prime_characterization}, $P$ is prime. As $IJ\subseteq I\cap J\subseteq P$, either $I\subseteq P$ or $J\subseteq P$. If $I\subseteq P$, then $I\subseteq J$, so $I=I\cap J$ is radical. If $J\subseteq P$, then $J=P$ is prime.\\
(2) By part~1, $I$ is prime or $J$ is prime. By symmetry, w.l.o.g.\ $I$ is prime. Inspecting the proof of part~1, since $I\subseteq J$ is now excluded by hypothesis, we conclude that also $J$ is prime.\\
(3) Since $m(J)\subseteq m(I+J)\subseteq \radpart{(I+J)}$ and $m(J)$ is prime by proposition \ref{prop_minimal_primes}, either $I+\radpart{(I+J)}\subseteq J$ or $J\subseteq I+\radpart{(I+J)}$ by theorem \ref{thm_pseudoprime}. In the first case, $I\subseteq I+\radpart{(I+J)}\subseteq J$, which contradicts the hypotheses. So $J\subseteq I+\radpart{(I+J)}$. %, hence $I+J = I + \radpart{(I+J)}$.
As $I+m(J)$, $\radpart{(I+J)}$ contain the prime ideal $m(J)$, either $I+m(J)\subseteq \radpart{(I+J)}$ or $\radpart{(I+J)}\subseteq I+m(J)$ by theorem \ref{thm_pseudoprime}. In the first case, $I+J\subseteq I+ \radpart{(I+J)}\subseteq \radpart{(I+J)}$, hence $I+J$ is radical and pseudoprime, hence prime or equal to $\GenK$ by theorem \ref{thm_prime_characterization}. Otherwise, $I+J\subseteq I + \radpart{(I+J)}\subseteq I+m(J)$, so $I+J=I+m(J)$. The statement follows by proposition \ref{prop_minimal_primes}(4).\\
(4) If $I+J=I+m(J)$, then $I$ and $m(J)$ are radical, hence $I+J$ is radical and pseudoprime, so $I+J$ is prime or $I+J=\GenK$ by theorem \ref{thm_prime_characterization}. The result follows by part~3.\\
(5) This holds in any commutative $l$-ring with $1$ in which every ideal is an $l$-ideal \cite[4.13]{Subramanian67}.
%Let $x\in\GenK\setminus I$. By Zorn's lemma, there exists an ideal $P$ with $I\subseteq P$ and $x\notin P$. We show that $P$ is pseudoprime.\\
%Otherwise, there would be $S\subseteq(0,1)$ with $e_S\notin P$ and $e_{\co S}\notin P$. By the maximality of $P$, $x\in P+e_S\GenK$ and $x\in P+ e_{\co S}\GenK$. Hence $xe_S, xe_{\co S}\in P$, and $x=xe_S + xe_{\co S}\in P$, a contradiction.
\end{proof}

\section{$z$-ideals in $\GenK$}
As the notion of $z$-ideal in the ring $\Cnt{}(X)$ of continuous functions on a topological space $X$ can be expressed by a purely algebraic condition \cite[4A]{GJ}, G.~Mason \cite{Mason73} used this condition to define a $z$-ideal of any commutative ring $R$ with $1$. 
\begin{df}
Denoting by $\Max(a)=\{M$ max.\ ideals of $R: a\in M\}$, $I\ideal R$ is a $z$-ideal iff
%\[(\forall a\in R)(\forall b\in I)(\Max(b)\subseteq \Max(a)\implies a\in I),\]
%iff
\[(\forall a\in R)(\forall b\in I)(\Max(a) = \Max(b)\implies a\in I).\]
\end{df}
In $\GenK$, the zeroes of a moderate net in $\K^{(0,1)}$ don't give rise to a definition of zeroes of an element $x\in\GenK$, because they depend on representatives, so the usual notion of $z$-ideal in $\Cnt{}(X)$ cannot directly be used in this context. But a natural generalization of the notion of zeroes of an element presents itself. We denote
\begin{align*}
\mathcal S_1&=\{S\subseteq(0,1): 0\in\overline S\}=\{S\subseteq(0,1): e_S\ne 0\}
=\mathcal S\cup\{S\subseteq(0,1): e_S=1\}.
\end{align*}

\begin{df}
Let $x\in \GenK$ and $S\in\mathcal S_1$. Then $x=0$ w.r.t.\ $S$ iff $e_S x =0$.
\end{df}
So subsets of $(0,1)$ take the role of `zeroes'. In this respect, it is natural to define the zero-set of $x\in\GenK$ by
$Z(x)=\{S\in\mathcal S_1: e_S x = 0\}$.\\
Sometimes, it is useful to formulate results about zeroes in terms of invertibility:
\begin{df}
$x\in\GenK$ is invertible w.r.t.\ $S\in\mathcal S_1$ iff $(\exists y\in\GenK)(xy=e_S)$.
\end{df}
Similarly, we define
$\Inv(x)=\{S\in\mathcal S_1: x \text{ is invertible w.r.t. }S\}$.\\
In analogy with $\Cnt{}(X)$, we would then say that $I\ideal \GenK$ is a $z$-ideal iff
\[(\forall a\in \GenK)(\forall b\in I)(Z(a)=Z(b)\implies a\in I).\]
Fortunately, these two notions coincide, as we will see in theorem \ref{thm_z-ideals_equiv}.

First, we collect some elementary properties of $Z$ and $\Inv$.
\begin{lemma}\label{lemma_zero-inv}
Let $n\in\N$. Let $a$, $b$, $c$, $a_n$ denote elements of $\GenK$; $(a_\eps)_\eps$, $(b_\eps)_\eps$ representatives of $a$, $b$; $S$, $T$ elements of $\mathcal S_1$.
\begin{enumerate}
\item $S, T\in Z(a)\implies S\cup T\in Z(a)$\\
$S,T\in \Inv(a)\implies S\cup T\in\Inv(a)$
\item $S\in Z(a)$, $T\subseteq S \implies T\in Z(a)$\\$S\in \Inv(a)$, $T\subseteq S \implies T\in \Inv(a)$
\item
\begin{enumerate}
	\item $S\in Z(a)$ iff $(\forall m\in\N)(\exists \eta >0)(\forall\eps\in S\cap (0,\eta))(\abs{a_\eps}\le\eps^m)$
	\item $S\in \Inv(a)$ iff $(\exists m\in\N)(\exists \eta >0)(\forall\eps\in S\cap (0,\eta))(\abs{a_\eps}\ge\eps^m)$
\end{enumerate}
\item
\begin{enumerate}
	\item $S\in Z(a)$ iff $(\forall T\in\mathcal S_1, T\subseteq S)(T\notin\Inv(a))$
	\item $S\in\Inv(a)$ iff $(\forall T\in\mathcal S_1, T\subseteq S)(T\notin Z(a))$
\end{enumerate}
\item $Z(a)\subseteq Z(b)$ iff $\Inv(b)\subseteq\Inv(a)$
\item $Z(a)=Z(a^n)=Z(\abs{a})$ and $\Inv(a)=\Inv(a^n)=\Inv(\abs{a})$
\item
\begin{enumerate}
	\item if $b\in a\GenK$, then $Z(a)\subseteq Z(b)$
	\item if $\abs{b}\le \abs{a}$, then $Z(a)\subseteq Z(b)$
	\item if $a\GenK + b\GenK= c\GenK$, then $Z(a)\cap Z(b)= Z(c)$
\end{enumerate}
\item $S\in Z(a)$ iff $a\in e_{\co S}\GenK$\\
$S\in\Inv(a)$ iff $e_S\in a\GenK$
\item $Z(e_S)=\{T\in \mathcal S_1: 0\notin \overline{S\cap T}\}=\{T\in \mathcal S_1: (\exists \eta > 0)(S\cap T\cap (0,\eta)=\emptyset)\}$\\
$\Inv(e_S)= \{T\in\mathcal S_1: 0\notin \overline{T\setminus S}\}=\{T\in\mathcal S_1: (\exists \eta > 0)(T\cap (0,\eta)\subseteq S)\}$\\
In particular, $\Inv(e_S)= Z(e_{\co S})$
\item
	\begin{enumerate}
	\item $\Inv(a)\cap\Inv(b)=\Inv(ab)$
	\item $Z(a)\cap Z(b)=Z(\abs{a}^2+\abs{b}^2)=Z(\abs{a}+\abs{b})=Z(\abs{a}\vee\abs{b})$
	\item if $A\subseteq \GenK$ and $\sup_{a\in A}\abs{a}$ exists, then $\bigcap_{a\in A} Z(a)=Z\big(\sup_{a\in A}\abs{a}\big)$
	\item if $a_n\to 0$, then $\bigcap_{n\in\N} Z(a_n)=Z\big(\sum_{n\in\N} \abs{a_n}^2\big)=Z\big(\sum_{n\in\N} \abs{a_n}\big)$
	\end{enumerate}
\item if $S\in Z(ab)$, then $S\in Z(a)$ or $S\in Z(b)$ or $(\exists T\subseteq S)(T\in Z(a)\,\,\&\,\, S\setminus T\in Z(b))$\\
if $S\in \Inv(a+b)$, then $S\in\Inv(a)$ or $S\in\Inv(b)$ or $(\exists T\subseteq S)(T\in\Inv(a)\,\,\&\,\, S\setminus T\in\Inv(b))$
\item $ab=0$ iff $\Inv(a)\cap \Inv(b)=\emptyset$ iff $\Inv(a)\subseteq Z(b)$
\item $Z(a)\subseteq Z(b)$ iff $\Ann(a)\subseteq \Ann(b)$
\item if $I\ideal\GenK$ is closed, then $\{Z(x): x\in I\}$ is closed under countable intersection.
\end{enumerate}
\end{lemma}
\begin{proof}
(1), (2) Immediate.\\
(3a) $\Leftarrow$: on representatives, $(a e_S)_\eps=\begin{cases}a_\eps,&\eps\in S\\0,&\text{otherwise.}\end{cases}$\\
So $(a e_S)_\eps$ is a negligible net, i.e., $a e_S=0$.\\
(3b) $\Leftarrow$: let $b_\eps=\begin{cases}1/a_\eps,&\eps\in S\cap (0,\eta)\\0,&\text{otherwise.}\end{cases}$\\
Then $(b_\eps)_\eps$ is a moderate net, so it represents $b\in\GenK$ for which $ab=e_S$.\\
(4a) $\Rightarrow$: if $ae_S=0$, then for $T\in\mathcal S_1$ and $b\in\GenK$, $ab=e_T$ implies that $0=abe_S=e_Te_S=e_{T\cap S}$. So $T\not\subseteq S$.\\
(4b) $\Rightarrow$: if $\exists b\in\GenK$ such that $ab=e_S$, then for $T\in\mathcal S_1$, $ae_T=0$ implies that $0=abe_T=e_Se_T=e_{S\cap T}$. So $T\not\subseteq S$.\\
(3,4 a) We prove the remaining implication. Suppose $(\exists m\in\N)(\forall\eta>0)(\exists \eps\in S\cap (0,\eta))(\abs{a_\eps}>\eps^m)$, then there exists a sequence $(\eps_n)_{n\in\N}$, with each $\eps_n\in S$ and $\abs{a_{\eps_n}}>\eps^m$, such that $\eps_n\to 0$. So $\{\eps_n:n\in\N\}\in\Inv(a)$.\\
(3,4 b) We prove the remaining implication. Suppose $(\forall m\in\N)(\forall\eta>0)(\exists \eps\in S\cap (0,\eta))(\abs{a_\eps}<\eps^m)$, then there exists a sequence $(\eps_n)_{n\in\N}$, with each $\eps_n\in S$ and $\abs{a_{\eps_n}}<\eps^n$, such that $\eps_n\to 0$. So $\{\eps_n:n\in\N\}\in Z(a)$.\\
(5) If $S\in\Inv(b)\setminus\Inv(a)$, then by part 4, there exists $T\subseteq S$, $T\in\mathcal S_1$, such that $T\in Z(a)\setminus Z(b)$. The converse implication follows dually.\\
(6) $a e_S = 0$ iff $a^n e_S = (a e_S)^n = 0$ iff $\abs{a}e_S = 0$. The second assertion follows from part 5.\\
(7a) If $ac=b$ and $ae_S =0$, then $be_S=ace_S=0$.\\
(7b) By order convexity of ideals in $\GenK$.\\
%Let $S\in Z(a)=Z(\abs{a})$, i.e., $\abs{a}e_S= 0$. Then $0\le \abs{b}e_S\le\abs{a}e_S=0$, so $S\in Z(b)=Z(\abs{b})$.\\
(7c) If $ae_S=be_S=0$ and $c=xa+yb$, then also $ce_S=0$.\\
The converse inclusion follows by part~7a.\\
(8) Let $S\in Z(a)$, so $a e_S = a(1-e_{\co S}) = 0$. So $a = a e_{\co S}$. Conversely, let $b\in\GenK$ and $a = b e_{\co S}$. Then $a e_S= b e_{\co S} e_S = 0$. Notice that the statement is equivalent to $\GenK e_S=\{x\in\tilde K: e_{\co S}\in Z(x)\}$.\\
(9) $e_Se_T=e_{S\cap T}=0$ iff $0\notin\overline{S\cap T}$.\\
By part 8, $T\in\Inv(e_S)$ iff $e_T\in\GenK e_S$ iff $e_{\co S}\in Z(e_T)$ iff $0\notin\overline{T\setminus S}$.\\
(10a) If $a x=e_S$ and $b y=e_S$, then $(ab)(xy)=e_S$. Conversely, if $(ab)c=e_S$, then $a(bc)= b(ac)= e_S$.\\
(10b) As $\abs{a}^2+\abs{b}^2\ge \abs{a}^2\ge 0$, $(\abs{a}^2+\abs{b}^2)e_S=0$ implies that $ae_S=0$ and, similarly, $be_S=0$. The converse follows from $(\abs{a}^2+\abs{b}^2)e_S=\abs{ae_S}^2+\abs{be_S}^2$.\\
The remaining equalities follow from part 7c and lemma \ref{lemma_principal_ideals}.\\
(10c) $\subseteq$: if $a e_S=0$, $\forall a\in A$, then $\sup_{a\in A}\abs{a}\ge (\sup_{a\in A}\abs{a})e_{\co S}\ge \abs{a} e_{\co S}= \abs{a}$, $\forall a\in A$. Hence $\sup_{a\in A}\abs{a}=(\sup_{a\in A}\abs{a}) e_{\co S}$, i.e., $(\sup_{a\in A}\abs{a})e_S=0$.\\
$\supseteq$: by part 7b.\\
(10d) as $\GenR$ is an ultrametric topology, $a_n\to 0$ implies that $\sum_n\abs{a_n}$ converges (and $a_n\to 0$ implies that $a_n^2\to 0$, hence also $\sum_n\abs{a_n}^2$ converges).\\
$\subseteq$: if $a_n e_S=0$, $\forall n$, then $\big(\sum \abs{a_n}\big)e_S=\sum \big(\abs{a_n}e_S\big)=0$. Similarly, $\big(\sum\abs{a_n}^2\big)e_S=0$.\\
$\supseteq$: by part 7b.\\
(11) Fix representatives $(a_\eps)_\eps$ of $a$ and $(b_\eps)_\eps$ of $b$. Let $T=\{\eps\in S: \abs{a_\eps}\le\abs{b_\eps}\}$.\\
$0\le (\abs{a}e_T)^2\le\abs{a}\abs{b}e_T=0$, so $ae_T=0$; similarly $be_{S\setminus T}=0$. If $e_T=e_S$, $S\in Z(a)$; if $e_T=0$, $S\in Z(b)$; otherwise, $e_T\ne 0$ and $e_{S\setminus T}\ne 0$, hence $T\in Z(a)$ and $S\setminus T\in Z(b)$.\\
Let $(a+b)c=e_S$. If $T\in\mathcal S_1$ and $T\notin\Inv(b)$, then by part~4 there exists $U\in{\mathcal S}_1$ with $U\subseteq T$ such that $be_U=0$. Further, $0\le \abs{a}e_U\le\abs{b}e_U =0$, so $a e_U=0$, and $e_U= (a+b)ce_U= 0$, a contradiction. So $e_T=0$ or $T\in\Inv(b)$. Similarly, $e_{S\setminus T}=0$ or $S\setminus T\in\Inv(a)$.\\
(12) By part 10, $ab=0$ iff $\Inv(ab)=\Inv(a)\cap\Inv(b)=\emptyset$. Clearly, if $\Inv(a)\subseteq Z(b)$, then $\Inv(a)\cap\Inv(b)\subseteq Z(b)\cap\Inv(b)=\emptyset$. Conversely, if $ab=0$, $S\in\mathcal S_1$ and $ac=e_S$, then $b e_S = b a c= 0$.\\
(13) $\Rightarrow$: let $x\in\Ann(a)$, i.e., $ax=0$. By part 12, $\Inv(x)\subseteq Z(a)\subseteq Z(b)$, so $bx=0$, i.e., $x\in\Ann(b)$.\\
$\Leftarrow$: Let $S\in \mathcal S_1$. $S\in Z(a)$ iff $ae_S=0$ iff $e_S\in\Ann(a)$.\\
(14) Let $a_n\in I$, $\forall n\in\N$. As $0\le\abs{a_n}\wedge \alpha^n\le\alpha^n\to 0$, and $\GenK$ is an ultrametric topology, $\sum_{n\in\N}\abs{a_n}\wedge \alpha^n$ converges to some $a\in \GenK$. By absolute order convexity of ideals, $a\in \overline I=I$. By part 10, $Z(a)=\bigcap_n Z(\abs{a_n}\wedge\alpha^n)=\bigcap_n Z(a_n)$.
\end{proof}

Now we are ready to prove the equivalence of the definitions of $z$-ideal in $\GenK$.
\begin{thm}\label{thm_z-ideals_equiv}
Let $a, b\in\GenK$. Then $\Max(a)\subseteq \Max(b)\iff Z(a)\subseteq Z(b)$.
\end{thm}
\begin{proof}
$\Rightarrow$: let $S\in Z(a)\setminus Z(b)$, i.e., $ae_S=0$ and $be_S\ne 0$. By lemma \ref{lemma_zero-inv}, there exists $T\subseteq S$ with $e_T\ne 0$ such that $b$ is invertible w.r.t.\ $T$. As $e_T\ne 0$, $e_{\co T}$ is not invertible, so there exists a maximal ideal $M$ containing $e_{\co T}$. Since $ae_T= a e_S e_T=0$, $a = a e_{\co T}\in M$. As $b$ is invertible w.r.t.\ $T$ and $e_T\notin M$, also $b\notin M$. So $M\in \Max(a)\setminus \Max(b)$.\\
$\Leftarrow$: let $M\in \Max(a)\setminus \Max(b)$, so $a\in M$ and $b\notin M$. By the maximality of $M$, $M + b\GenK=\GenK$. Let $m\in M$ and $c\in \GenK$ such that $m+bc=1$. By lemma \ref{lemma_zero-inv}(11), either $m$ is invertible (which is impossible since $m\in M$), or there exists $T\subseteq (0,1)$ with $e_T\ne 0$ such that $T\in\Inv(bc)\subseteq\Inv(b)$ and $e_{\co T}\in m\GenK\subseteq M$. As $a\in M$ and $e_T\notin M$, $e_T\notin a\GenK$, i.e., $T\notin\Inv(a)$. So $\Inv(b)\not\subseteq\Inv(a)$, i.e., $Z(a)\not\subseteq Z(b)$ by lemma \ref{lemma_zero-inv}.
\end{proof}

Consequently, the two notions of $z$-ideal coincide.
We can use the general theory of $z$-ideals to obtain some of their properties. Some of the statements are proven for $z$-ideals in $\Cnt{}(X)$ in \cite{Azarpanah07}, \cite{Mason80}.
\begin{prop}\label{prop_z-closure}\leavevmode
\begin{enumerate}
\item For $I\ideal\GenK$,
\begin{align*}
\zclosure{I}:&=\{x\in\GenK: (\exists a\in I)(Z(x)= Z(a))\} =\{x\in\GenK: (\exists a\in I)(Z(x)\supseteq Z(a))\}\\
&=\{x\in\GenK: (\exists a\in I)(\Inv(x)= \Inv(a))\}
=\{x\in\GenK: (\exists a\in I)(\Inv(x)\subseteq \Inv(a))\}\\
&=\{x\in\GenK: (\exists a\in I)(\Max(x)= \Max(a))\}
=\{x\in\GenK: (\exists a\in I)(\Max(x)\supseteq \Max(a))\}
\end{align*}
is the smallest $z$-ideal containing $I$. We call it the $z$-closure of $I$. $I$ is a $z$-ideal iff $I=\zclosure I$.
\item For $I\ideal\GenK$, $I\subseteq\rad I\subseteq\zclosure{I}$. Hence $\zclosure{(\rad I)}=\zclosure I$ and every $z$-ideal is radical. A (proper) $z$-ideal is prime iff it is pseudoprime.
%\item Let $I$ be a $z$-ideal and $P$ be minimal in the class of prime ideals containing $I$, then $P$ is a $z$-ideal.
\item If $I\ideal\GenK$ and $J$ is a $z$-ideal, then also $(J:I)=\{x\in\GenK: xI\subseteq J\}$ is a $z$-ideal.
\end{enumerate}
\end{prop}
\begin{proof}
(1) Let $M$ be a maximal ideal in $\GenK$, let $a$, $b\in \GenK$ and let $\abs{a}^2+\abs{b}^2\in M$. By lemma \ref{lemma_principal_ideals}, $a^2\in M$ and $b^2\in M$. By the maximality of $M$, $M=\rad M$, hence $a,b\in M$. So $\Max(a)\cap\Max(b)=\Max(\abs{a}^2+\abs{b}^2)$. Under this condition, $\zclosure I=\{x\in\GenK: (\exists a\in I)(\Max(x)\supseteq \Max(a))\}$ \cite[Prop.~1.13]{Mason73}. As in any commutative ring with $1$, $\Max(a)\subseteq \Max(x)\iff \Max(x)=\Max(ax)$, so also $\zclosure I=\{x\in\GenK: (\exists a\in I)(\Max(x)= \Max(a))\}$. The other equalities follow from theorem \ref{thm_z-ideals_equiv} and lemma \ref{lemma_zero-inv}.\\
(2) $I\subseteq\rad I\subseteq \zclosure I$ holds in any commutative ring with $1$ \cite{Mason73}. The last assertion then follows from theorem \ref{thm_prime_characterization}.\\
(3) holds in any commutative ring with $1$ \cite[Prop.~1.3]{Mason73}.
\end{proof}

\begin{prop}\label{prop_z_part}\leavevmode
\begin{enumerate}
\item For a family $(I_\lambda)_{\lambda\in\Lambda}$ of ideals $I_\lambda\ideal\GenK$, $\zclosure{(\sum_{\lambda\in\Lambda} I_\lambda)}=\sum_{\lambda\in\Lambda} \zclosure{(I_\lambda)}$. In particular, the sum of a family of $z$-ideals is a $z$-ideal.
\item For $I,J\ideal\GenK$, $\zclosure{I}\cap \zclosure{J}=\zclosure{(I\cap J)}$.
\item For $I\ideal\GenK$, $\zpart I:=\{x\in\GenK: \zclosure{(x\GenK)}\subseteq I\}$ is the largest $z$-ideal contained in $I$. We call it the $z$-part of $I$. $I$ is a $z$-ideal iff $I=\zpart I$.
\item For a family $(I_\lambda)_{\lambda\in\Lambda}$ of ideals $I_\lambda\ideal\GenK$, $\bigcap_{\lambda\in\Lambda} \zpart I_\lambda=\zpart{(\bigcap_{\lambda\in\Lambda} I_\lambda)}$. In particular, the intersection of a family of $z$-ideals is a $z$-ideal.
\item For $I\ideal\GenK$, $m(I)\subseteq\zpart I\subseteq \radpart I\subseteq I$. In particular, every pure ideal of $\GenK$ is a $z$-ideal. If $I\idealproper\GenK$ is pseudoprime, then $\zpart I$ is prime.
\end{enumerate}
\end{prop}
\begin{proof}
(1) We first show that $\zclosure{(I+J)}\subseteq\zclosure I +\zclosure J$, $\forall I,J\ideal\GenK$.\\
Let $a\in \zclosure{(I+J)}$. So $Z(a)=Z(\alpha + \beta)$, for some $\alpha\in I$, $\beta\in J$. Choose representatives $(\alpha_\eps)_\eps$ of $\alpha$ and $(\beta_\eps)_\eps$ of $\beta$. Let $S=\{\eps\in (0,1): \abs{\alpha_\eps}\ge\abs{\beta_\eps}\}$. Then $a=a e_S + a e_{\co S}$. We show that $a e_S\in \zclosure I$ and $a e_{\co S}\in \zclosure J$.\\
Let $T\in Z(\alpha)$, i.e., $\alpha e_{T}=0$. In particular, $\alpha e_{S\cap T}=0$. As $0\le \abs{\beta}e_S\le \abs{\alpha}e_S$, also $\beta e_{S \cap T}=0$. So $S\cap T\in Z(\alpha +\beta)=Z(a)$, i.e., $a e_S e_T = 0$ and $T\in Z(a e_S)$. So $Z(\alpha)\subseteq Z(ae_S)$. Similarly, $a e_{\co S}\in \zclosure J$.\\
Now for a family $(I_\lambda)_{\lambda\in\Lambda}$ of ideals, clearly $\sum_{\lambda\in\Lambda} \zclosure{(I_\lambda)}\subseteq \zclosure{(\sum_{\lambda\in\Lambda} I_\lambda)}$. Conversely, if $x\in\zclosure{(\sum_{\lambda\in\Lambda} I_\lambda)}$, then for some finite $\Lambda_0\subseteq\Lambda$, $x\in\zclosure{(\sum_{\lambda\in\Lambda_0} I_\lambda)} \subseteq \sum_{\lambda\in\Lambda_0} \zclosure{(I_\lambda)} \subseteq \sum_{\lambda\in\Lambda} \zclosure{(I_\lambda)}$.\\
(2) holds for rings satisfying the condition that was verified in the proof of prop.\ \ref{prop_z-closure}(1) \cite[Prop.~1.13]{Mason73}.\\%Alternatively, use $\Inv(ab)=\Inv(a)\cap\Inv(b)$
(3) $\zpart I\ideal\GenK$, since for $x,y\in\GenK$, $\zclosure{((x+y)\GenK)}\subseteq \zclosure{(x\GenK)} + \zclosure{(y\GenK)}$ and $\zclosure{(xy\GenK)} \subseteq \zclosure{(x\GenK)}$. Let $x\in \zpart I$ and $Z(x)=Z(y)$. Then $\zclosure{(x\GenK)}=\zclosure{(y\GenK)}$, so also $y\in \zpart I$. Hence $\zpart I$ is a $z$-ideal contained in $I$. Now let $J$ be a $z$-ideal contained in $I$. Then for each $x\in J$, $\zclosure{(x\GenK)}\subseteq\zclosure J=J\subseteq I$. Hence $J\subseteq \zpart I$.\\
(4) $x\in \bigcap_{\lambda\in\Lambda} \zpart I_\lambda$ iff $(\forall\lambda\in\Lambda)(\zclosure{(x\GenK)}\subseteq I_\lambda)$ iff $x\in \zpart{\big(\bigcap_{\lambda\in\Lambda} I_\lambda\big)}$ (cf.\ also \cite[Lemma~3.4]{Mason80}).\\
(5) Let $x\in m(I)$. Then $x=xe_S$, for some $e_S\in I$. Let $y\in \zclosure{(x\GenK)}$. As $x e_{\co S}=0$, also $y e_{\co S}=0$, i.e., $y=y e_S\in I$. Hence $x\in \zpart I$, and $m(I)\subseteq \zpart I$. By proposition \ref{prop_z-closure}, $\rad{x\GenK}\subseteq \zclosure{(x\GenK)}$ for each $x\in\GenK$, so $\zpart I\subseteq \radpart I$. If $I$ is pseudoprime, then $m(I)$ is prime by proposition \ref{prop_minimal_primes}, hence $\zpart I\supseteq m(I)$ is pseudoprime and radical by proposition \ref{prop_z-closure}, hence prime by theorem \ref{thm_prime_characterization}.
\end{proof}

\begin{prop}\label{prop_z-ideals_advanced}Let $I,J\idealproper\GenK$.
\begin{enumerate}
\item If $J$ is a $z$-ideal and $J\subseteq\rad I$, then $J\subseteq I$.
\item $I$ is a $z$-ideal iff $\rad I$ is a $z$-ideal.
\item If $I$ and $J$ are pseudoprime and $I\cap J$ is a $z$-ideal, then $I$ is a $z$-ideal or $J$ is a prime $z$-ideal.
\item If $I$ and $J$ are pseudoprime, $I\cap J$ is a $z$-ideal, $I\not\subseteq J$ and $J\not\subseteq I$, then $I$, $J$ are prime $z$-ideals.
\item Let $J$ be pseudoprime and $I\not\subseteq J$. If $m(I)\subseteq m(J)$, then $I+J$ is a prime $z$-ideal or $I+J=I+m(J)$; if $m(I)\not\subseteq m(J)$, then $I+J=\GenK$.
\item Let $I$ be a $z$-ideal, $J$ pseudoprime and $I\not\subseteq J$. If $m(I)\subseteq m(J)$, then $I+J$ is a prime $z$-ideal. If $m(I)\not\subseteq m(J)$, then $I+J=\GenK$.
\item Every $z$-ideal of $\GenK$ is an intersection of prime z-ideals.
\end{enumerate}
\end{prop}
\begin{proof}
(1) Let $x\in J$. As $x\in\GenK$, $y=\abs{x}\alpha^N\le 1$, for some $N\in\N$. Let $a=\sum_{n=1}^\infty \alpha^n \sqrt[n]{y}$. As $0\le\sqrt[n]y\le 1$, $\forall n\in\N$, the sum converges. Further, if $S\in Z(x)=Z(y)$, then $a e_S=\sum_{n=1}^\infty \alpha^n \sqrt[n]{y}e_S=0$, so also $S\in Z(a)$. Hence $a\in \zclosure{J}= J \subseteq\rad I$. So there exists $n\in\N$ such that $a^n\in I$. As $a\ge \alpha^n\sqrt[n]y\ge 0$, $\sqrt[n]y\in a\GenK$, hence $x\in y\GenK\subseteq a^n\GenK\subseteq I$.\\
(2) $\Rightarrow$: by proposition \ref{prop_z-closure}. $\Leftarrow$: by part~1.\\
(3--6): by proposition \ref{prop_z_part}(5), this is completely analogous to the proof of proposition \ref{prop_radical_advanced}(1--4) (using the $z$-part instead of the radical part).\\
(7) This holds in any commutative ring with $1$ \cite[1.0--1.1]{Mason73}.
%(7) Let $I$ be a $z$-ideal of $\GenK$. If $x\notin I$, then consider $A=\{y\in\GenK: \Inv(x)=\Inv(y)\}$. By lemma \ref{lemma_zero-inv}, $A$ is closed under multiplication, so by proposition~\ref{prop_prime_constructor}, there exists $P\idealproper\GenK$ prime with $I\subseteq P$ and $P\cap A=\emptyset$. So $x\notin \zclosure P$, hence $x\notin \bigcap_{I\subseteq P\atop P\text{ prime}} \zclosure P$. So $I = \bigcap_{I\subseteq P\atop P\text{ prime}} \zclosure P$, which is an intersection of prime $z$-ideals.
\end{proof}

\section{Closed ideals in $\GenK$}
With the following definition, we want to formalize a number of methods applied at the level of representatives in \cite{AJ2001}.
\begin{df}
Let $a\in\GenK$. Fix a representative $(a_\eps)_\eps$ of $a$. Let for each $n\in\N$, $L_n=\{\eps\in (0,1): \abs{a_\eps}\ge \eps^n\}$. Then we call $(L_n)_{n\in\N}$ a \defstyle{sequence of level sets} for $a$.
\end{df}
As this definition depends on representatives, a sequence of level sets is not unique. However, many useful properties do not depend on the chosen representative. E.g., each sequence of level sets is increasing (w.r.t.\ $\subseteq$) and $a= \lim_n a e_{L_n}$ for each sequence of level sets $(L_n)_{n\in\N}$ for $a$. Some further properties are given in the next proposition.
\begin{prop}\label{prop_level_sets}
Let $a\in\GenK$ and let $(L_n)_{n\in\N}$ be a sequence of level sets for $a$.
\begin{enumerate}
\item $\Inv(a)=\{S\in{\mathcal S}_1: (\exists m\in\N)(\exists \eta >0)(S\cap (0,\eta)\subseteq L_m)\}$.
\item $m(a\GenK)=\lspan{e_S: S\in\Inv(a)}=\lspan{e_{L_n}: n\in\N}$.\\
In particular, $m(a\GenK)$ is generated by a countable number of idempotent elements.
\item $\Ann(a)=\bigcap_{n\in\N} e_{\co L_n}\GenK$.
%\item Let $I_\lambda\ideal\GenK$, for each ${\lambda\in\Lambda}$. Then $\sum_{\lambda\in\Lambda} m(I_\lambda)=m\Big(\sum_{\lambda\in\Lambda} I_\lambda\Big)$.
\item Let $I\ideal\GenK$ be countably generated. Then $m(I)$ is the pure part of a principal ideal.
\end{enumerate}
\end{prop}
\begin{proof}
(1) By lemma \ref{lemma_zero-inv}.\\
(2) If $e_S\in m(a\GenK)$, for some $S\subseteq (0,1)$ with $e_S\ne 0$, then $S\in \Inv(a)$, so $m(a\GenK)\subseteq \lspan{e_S: S\in\Inv(a)}$.\\
If $S\in\Inv(a)$, then, by part 1, $e_S= e_S e_{L_n}$ for some $n\in\N$, so $\lspan{e_S: S\in\Inv(a)}\subseteq \lspan{e_{L_n}: n\in\N}$.\\
By part 1, $L_n\in\Inv(a)$ if $e_{L_n}\ne 0$. So $e_{L_n}\in a\GenK$, $\forall n$. So $\lspan{e_{L_n}: n\in\N}\subseteq m(a\GenK)$.\\
(3) $\subseteq$: let $x\in \Ann(a)$. By part 2, $e_{L_n}\in a\GenK$, $\forall n$. So $e_{L_n}x=0$, $\forall n$, i.e., $x= xe_{\co L_n}$, $\forall n$.\\
$\supseteq$: if $x\in e_{\co L_n}\GenK$, $\forall n$, then $x e_{L_n}=0$, $\forall n$. So $x a= \lim_n x a e_{L_n}=0$.\\
%(4) holds in any Gelfand ring (see preliminaries)
%(4) $\subseteq$: if $x=\sum_{j=1}^n x_j e_{S_j}$, with $e_{S_j}\in I_{\lambda_j}$, $\lambda_j\in\Lambda$, then $x\in m\big(\sum_\lambda I_\lambda\big)$.\\
%$\supseteq$: let $e_S= \sum_{j=1}^n x_j$, with $x_j\in I_{\lambda_j}$, $\lambda_j\in\Lambda$. As $S\in\Inv\big(\sum_{j=1}^n x_j\big)$, by lemma \ref{lemma_zero-inv}, there exists a partition $\{T_1,\ldots, T_n\}$ of $S$ such that $T_j\in\Inv(x_j)$, for each $j$. So $e_{T_j}\in m(I_{\lambda_j})$, for each $j$ and $e_S=\sum_{j=1}^n e_{T_j}\in\sum_\lambda m(I_\lambda)$.\\
(4) Let $I=\lspan{a_n: n\in\N}$ ($a_n\in\GenK$). Let $(L_{n,m})_{m\in\N}$ be a sequence of level sets for each $a_n$. By part 2 and the fact that $\GenK$ is a Gelfand ring, $m(I)=m(\sum_n a_n\GenK)=\sum_n m(a_n\GenK)=\lspan{e_{L_{n,m}}: m,n\in \N}$. Now let $\kappa$: $\N^2\to\N$ a bijection, let $S_m=\bigcup_{\kappa(i,j)\le m}L_{i,j}$ and let $\beta$ be the element with representative $(\beta_\eps)_\eps$, where $\beta_\eps=\eps^m$, for $\eps\in S_m\setminus S_{m-1}$, and $\beta_\eps=0$, for $\eps\in(0,1)\setminus\bigcup_n S_n$. By part 2, $m(\beta\GenK)=\lspan{e_{S_n}:n\in\N}=\lspan{e_{L_{n,m}}: m,n\in \N}$.
\end{proof}

The (topologically) closed ideals can be characterized by means of $\Inv$.
\begin{thm}\label{thm_closed_characterization}
Let $I\idealproper \GenK$. Then
\begin{align*}
\overline I &= \{x\in\GenK: (\forall S\in \Inv(x))(e_S\in I)\}\\
&= \{x\in\GenK: m(x\GenK)\subseteq I\}\\
&=\{x\in\GenK: (\text{for each sequence $(L_n)_{n\in\N}$ of level sets for $x$})(\forall n\in\N)(e_{L_n}\in I)\}\\
&=\{x\in \GenK: (\forall n\in\N)(\exists S\subseteq(0,1))(e_{S}\in I \,\&\,\abs{x}e_{\co S}\le\alpha^n)\}\\
&=\overline{m(I)}.
\end{align*}
Every closed ideal in $\GenK$ is the closure of a pure ideal and $m(\overline I)=m(I)$.
%postponed!
%Every closed ideal in $\GenK$ is a z-ideal.
\end{thm}
\begin{proof}
(i) let $a\in \overline I$. If $a$ is invertible w.r.t.\ $S\in\mathcal S_1$, then there is a sharp neighbourhood $U$ of $a$ such that each element of $U$ is invertible w.r.t.\ $S$. As $a\in \overline I$, there exists $x\in U\cap I$ and $y\in\GenK$ such that $xy=e_S$. So $e_S\in I$.\\
(ii) If $e_S\in I$, $\forall S\in \Inv(x)$, then $m(x\GenK)=\lspan{e_S:S\in\Inv(x)}\subseteq I$ by proposition \ref{prop_level_sets}.\\
(iii) If $m(x\GenK)\subseteq I$ and $(L_n)_{n\in\N}$ is a sequence of level sets for $x$, then by proposition \ref{prop_level_sets}, $e_{L_n}\in m(x\GenK)\subseteq I$, $\forall n\in\N$.\\
(iv) If $(L_n)_{n\in\N}$ is a sequence of level sets for $x$, then $\abs{x}e_{\co L_n}\le\alpha^n$.\\
(v) if for each $n\in\N$, $e_{S_n}\in I$ and $\abs{x}e_{\co S_n}\le\alpha^n$, then $\lim_{n\to\infty} x e_{S_n}=x$. So $x\in\overline{m(I)}$.\\
(vi) $m(I)\subseteq I$, so $\overline{m(I)}\subseteq \overline I$.\\
If $e_S\in \overline I$, then $S\in\Inv(e_S)$, so by the characterization, $e_S\in I$. So $m(\overline I)\subseteq m(I)$.
%postponed:
%As membership of $x\in\GenK$ to $\overline I$ depends on $\Inv(x)$ only, any closed is ideal is a z-ideal.
\end{proof}

The characterization of maximal ideals proven in~\cite{AJ2001} can be viewed in the context of the characterization of the closed ideals.
\begin{thm}\label{thm_max_ideals}
\begin{enumerate}
\item Let $P\idealproper \GenK$ be a prime ideal. Then $\overline P$ is a maximal ideal.
\item The maximal ideals of $\GenK$ are exactly $\overline{g(\mathcal F)}$, where $\mathcal F\in P_*(\mathcal S)$.
\end{enumerate}
\end{thm}
\begin{proof}
(1) Let $x\in\GenK\setminus \overline P$. By theorem \ref{thm_closed_characterization}, there exists $S\in\Inv(x)$ such that $e_S\notin P$. As $P$ is prime, $e_{\co S}\in P$. As $e_S\in x\GenK$, $1= e_S + e_{\co S}\in \overline P+x\GenK$, so $\overline P +x\GenK = \GenK$. We conclude that $\overline P$ is maximal.\\
(2) Let $M$ be a maximal ideal of $\GenK$. As the set of invertible elements of $\GenK$ is open, $\overline M$ is a proper ideal, so $M=\overline M=\overline{m(M)}$ by the maximality and by theorem \ref{thm_closed_characterization}. Since $M$ is prime, $\overline{m(M)}= \overline{g(\mathcal F)}$ for some $\mathcal F\in P_*(\mathcal S)$ by theorem \ref{thm_pseudoprime}.
Conversely, if $\mathcal F\in P_*(\mathcal S)$, then $g(\mathcal F)$ is prime (cf.\ prop.~\ref{prop_minimal_primes}), so $\overline{g(\mathcal F)}$ is maximal by part~1.
\end{proof}

\begin{prop}\label{prop_intersection_of_max_ideals}
Let $I\idealproper \GenK$. Then $\overline I =\bigcap_{I\subseteq M\atop {M\text{ maximal}}}M$.\\
In particular, an ideal $I\idealproper\GenK$ is closed iff it is an intersection of maximal ideals.
\end{prop}
\begin{proof}
By theorem \ref{thm_max_ideals}, maximal ideals are closed, so $\overline I\subseteq\bigcap_{I\subseteq M\atop {M\text{ maximal}}} M$.\\
Conversely, if $x\notin \overline I$, by theorem \ref{thm_closed_characterization}, there exists $S\in\Inv(x)$ such that $e_S\notin I$. By proposition \ref{prop_prime_constructor}, there exists $P\idealproper\GenK$ prime with $I\subseteq P$ and $e_S\notin P$. By theorem \ref{thm_closed_characterization}, $e_S\notin \overline P$, and, as $S\in\Inv(x)$, $x\notin \overline P$. By theorem \ref{thm_max_ideals}, $\overline P$ is maximal, so $x\notin \bigcap_{I\subseteq M\atop {M\text{ maximal}}}M$. So $\bigcap_{I\subseteq M\atop {M\text{ maximal}}} M \subseteq \overline I$.
\end{proof}

\begin{cor}
G.~Mason \cite{Mason73} calls an ideal $I$ a strong $z$-ideal iff it is an intersection of maximal ideals. Hence the closed ideals of $\GenK$ are exactly the strong $z$-ideals of $\GenK$. As every strong $z$-ideal is a $z$-ideal, every closed ideal in $\GenK$ is a $z$-ideal.
\end{cor}

\begin{comment}
\begin{prop}
Let $I$ be a proper pseudoprime ideal. Then $\overline P=\{a\in\GenK: (\exists x\in \GenK)(ax-1\notin P)\}$.
\end{prop}
\begin{proof}
Can be checked elementary (generalizes \cite[8.42]{Borceux83}).
\end{proof}
\end{comment}

\begin{prop}\label{prop_countably_generated_ideals}
Let $I$ be a countably generated ideal of $\GenK$. Then $\overline I$ is the closure of a principal ideal.
\end{prop}
\begin{proof}
By proposition \ref{prop_level_sets}(4) and theorem \ref{thm_closed_characterization}.
\end{proof}

\begin{prop}
For $I$, $J$ $\ideal$ $\GenK$, $\overline I + \overline J =\overline{I+J}$. In particular, the sum of two closed ideals is a closed ideal.
\end{prop}
\begin{proof}
As $\overline I\subseteq\overline{I+J}$ and $\overline J\subseteq\overline{I+J}$, $\overline I + \overline J\subseteq\overline{I+J}$.\\
Conversely, let $x\in\overline{I+J}$. Fix a representative $(x_\eps)_\eps$ of $x$. We may suppose that there exists $N\in\N$ such that $\abs{x_\eps}<\eps^{-N}$, $\forall\eps$. Let $S_m=\{\eps\in(0,1): \eps^{m-N}\le\abs{x_\eps}<\eps^{m-N-1}\}$ for each $m\in\N$. By theorem~\ref{thm_closed_characterization}, $e_{S_m}\in I+J$, for each $m$. Let $e_{S_m}=a_m+b_m$, $a_m\in I$, $b_m\in J$. By lemma~\ref{lemma_zero-inv}(11), there exist $T_m\subseteq S_m$ such that %$T_m\in\Inv(a_m)$ and $S_m\setminus T_m\in \Inv(b_m)$. So
$e_{T_m}\in a_m\GenK\subseteq I$ and $e_{S_m\setminus T_m}\in b_m\GenK\subseteq J$.\\
We show that $\sum_{n\in\N} xe_{T_n}$ is a Cauchy sequence in $\GenK$, hence convergent.\\
With $L_m=\{\eps\in(0,1):\abs{x_\eps}\ge\eps^m\}$,
\[\abs[\Big]{\sum_{j=n+1}^{n+m}x e_{T_j}}\le\abs{x}\sum_{j=n+1}^{n+m}e_{S_j}
\le\abs{x}e_{\co L_{n-N}}\le\alpha^{n-N},\]
so $\sum_{n\in\N} xe_{T_n}=y$, for some $y\in\overline I$. Similarly, $\sum_{n\in\N} xe_{S_n\setminus T_n}=z$, for some $z\in\overline J$. So $x=\sum_{n\in\N} x e_{S_n}=y+z\in\overline I +\overline J$.
\end{proof}

\begin{thm}\label{thm_closed_fin_gen}
Let $I\ideal\GenK$ be a finitely generated (hence principal, by lemma~\ref{lemma_principal_ideals}) ideal.
\begin{enumerate}
\item The following are equivalent:
\begin{enumerate}
\item $I$ is closed
\item $I$ is a $z$-ideal
\item $I$ is radical
\item $I$ is pure
\item $(\exists S\subseteq (0,1))(I=e_S\GenK)$
\end{enumerate}
\item $\zclosure I=\overline I$%(although the latter ideal needn't be finitely generated).
\item $m(I)=\zpart I$.
\end{enumerate}
\end{thm}
\begin{proof}
$(1)$: $(a)\implies (b)$: by the corollary to proposition~\ref{prop_intersection_of_max_ideals}.\\
$(b)\implies (c)$: by proposition~\ref{prop_z-closure}.\\
$(c)\implies (e)$: as $I$ is principal, $I=a\GenK$ for some $a\in\GenK$. By proposition \ref{prop_radical}, $I$ is idempotent, hence $a = a^2 b$, for some $b\in\GenK$. So $ab=(ab)^2$; this implies that $ab=e_S$, for some $S\subseteq(0,1)$ \cite{AJOS2006}. Further, $a\in ab\GenK$ and $ab\in a\GenK$, so $I=e_S\GenK$.\\
%obsolete:
%So $a(1-ab)=0$. By lemma~\ref{zero divisors}, there exists $S\subseteq(0,1)$ such that $ae_{\co S}=0$ and $abe_S = e_S$. We show that $a\GenK=e_S\GenK$. As $a=a e_S$, $a\GenK\subseteq e_S\GenK$. As $e_S = ab e_S$, $e_S\GenK\subseteq a\GenK$.\\
$(e)\implies (a)$: if $x\in \overline I$, then $x=\lim_n x_n e_S$, so $x e_S =\lim_n x_n e_S= x$, and $x\in I$.\\
$(e)\implies (d)$: since any ideal generated by idempotents is pure.\\
$(d)\implies (b)$: by proposition \ref{prop_z_part}.\\
$(2)$: let $I=a\GenK$ and $x\in\overline{I}$. By theorem~\ref{thm_closed_characterization}, $\Inv(x)\subseteq\Inv(a)$, so $x\in \zclosure{(a\GenK)}$. The converse inclusion holds generally (theorem~\ref{thm_closed_characterization}).\\
$(3)$: Let $I=a\GenK$ and let $x\in\GenK \setminus m(I)$. Let $(L_n)_{n\in\N}$ be a sequence of level sets for $a$.
As each $e_{L_n}\in m(I)$, $x\ne x e_{L_n}$, i.e., $x e_{\co L_n}\ne 0$ for each $n\in\N$. Hence there exist $T_n\subseteq \co L_n$ with $e_{T_n}\ne 0$ such that $T_n\in\Inv(x)$. Fix a representative $(x_\eps)_\eps$ of $x$ and let
\[
y_\eps=
\begin{cases}
\eps^{n/2}, &\eps\in T_n\setminus \bigcup_{m< n} T_m,\quad n\in\N\\
x_\eps, &\text{otherwise}.
\end{cases}
\]
As $(y_\eps)_\eps$ is moderate, it represents some $y\in\GenK$. We show that $y\in \overline{x\GenK}$.\\
Let $(K_n)_{n\in\N}$ be a sequence of level sets for $x$ and let $m\in\N$. As $\lim_n x e_{\co K_n}=0$, we can find $N\in\N$ such that $\abs{x}e_{\co K_n}\le\alpha^m$, $\forall n\ge N$. As $T_n\in\Inv(x)$, $\forall n$, we can find by lemma \ref{prop_level_sets}, $M\in\N$ such that $e_{T_n}= e_{T_n} e_{K_M}$, forall $n\le m$. Let $U=\bigcup T_n$. Then
\[e_{\co K_M}e_U\le e_{\co K_M}(e_{T_1}+e_{T_2}+\cdots + e_{T_m} + e_{U\setminus(T_1\cup\cdots\cup T_m)})= e_{\co K_M} e_{U\setminus(T_1\cup\cdots\cup T_m)}.\]
Hence $\abs{y}e_{\co K_n}=\abs{y}e_{\co K_n} e_U + \abs{y}e_{\co K_n}e_{\co U}\le\alpha^{m/2} e_U+ \alpha^m e_{\co U}\le \alpha^{m/2}$, as soon as $n\ge M$, $n\ge N$.\\
Hence $y=\lim_n y e_{K_n}\in \overline{x\GenK}=\zclosure{(x\GenK)}$ by part~2. Should $x\in \zpart I$, then $y\in \zpart I\subseteq I$, so $\abs{y}\le \alpha^{-N}\abs a$, for some $N\in\N$. But $\abs{a}e_{T_n}\le \alpha^n e_{T_n}$, and $\abs{y}e_{T_n}\ge\alpha^{n/2}e_{T_n}$, $\forall n\in\N$, a contradiction. So $x\notin \zpart I$. The converse inclusion holds generally (proposition \ref{prop_z_part}).
\end{proof}
\begin{cor}
If $I\ideal\GenK$ is not closed, then $\overline{I}$ is not principal (hence not finitely generated by lemma \ref{lemma_principal_ideals}).
\end{cor}
\begin{proof}
Suppose $\overline{I}$ is principal. By theorem \ref{thm_closed_fin_gen}, $\overline{I}=e_S\GenK$, for some $S\subseteq(0,1)$. By theorem \ref{thm_closed_characterization}, $m(I)=e_S\GenK$. Hence $e_S\GenK\subseteq I\subseteq e_S\GenK$, and $I$ would be closed, a contradiction.
\end{proof}

\begin{rem}
As shown in \cite{AJOS2006}, not every principal ideal of $\GenK$ is generated by an idempotent. E.g., consider $\beta \GenK$, where $\beta$ is as in example \ref{ex_z-prime} below.
\end{rem}

A generator of an ideal satisfying the equivalent conditions of theorem \ref{thm_closed_fin_gen} can be described more explicitly (compare also with \cite[Lemma~4.23]{AJ2001}):
\begin{prop}\label{prop_edged_generator}
Let $a\in\GenK\setminus\{0\}$ and $(L_n)_{n\in\N}$ a sequence of level sets for $a$.\\
Then the following are equivalent:
\begin{enumerate}
\item $(\exists S\subseteq (0,1))$ $(a\GenK=e_S\GenK)$
\item $(\exists S\subseteq (0,1))$ $(S\in\Inv(a)\,\&\, \co S\in Z(a))$
\item $(L_n)_{n\in\N}$ is stationary, i.e., $(\exists N\in\N)(\forall n\ge N)(e_{L_n}=e_{L_N})$, or equivalently, $(\exists N\in\N)(\forall n\ge N)(\exists \eta>0)(L_n\cap (0,\eta) = L_N\cap (0,\eta))$.
\end{enumerate}
\end{prop}
\begin{proof}
$(1)\implies (2)$: by lemma \ref{lemma_zero-inv}(8).\\
$(2)\implies (3)$: as $S\in\Inv(a)$, by proposition \ref{prop_level_sets}, $S\cap (0,\eta)\subseteq {L_N}$, for some $N\in\N$ and $\eta\in(0,1)$, hence also $S\cap (0,\eta)\subseteq {L_n}$ for each $n\ge N$. As $\co S\in Z(a)$, by lemma \ref{lemma_zero-inv}(3), for each $n\ge N$, $\co S \cap(0,\eta_n)\subseteq {\co L_n}$, for some $\eta_n\in (0,1)$. So $L_n\subseteq S\cup [\eta_n,1)$ and $e_S\le e_{L_n}\le e_S$, so $e_{L_n}=e_S$, $\forall n\ge N$.\\
$(3)\implies(1)$: as $a=\lim_n a e_{L_n}$, $a=a e_{L_N}$, so $a\in e_{L_N}\GenK$. As $L_N\in\Inv(a)$, also $e_{L_N}\in a\GenK$. Hence $a\GenK= e_{L_N}\GenK$. 
\end{proof}

\section{Prime ideals, ultrafilters, nonstandard analysis}
\begin{lemma}
Let $\mathcal I_0:=\{(\eta,1): \eta\in (0,1)\}$.\\
(1) Let $\mathcal F\in P_*(\mathcal S)$. Then $\mathcal I:=\{S\subseteq(0,1): e_S\in g(\mathcal F)\}=\mathcal F \cup\{S\subseteq (0,1): 0\notin \overline S\}$ is a maximal cofilter on $(0,1)$ that contains $\mathcal I_0$.\\
(2) Conversely, if $\mathcal I$ is a maximal cofilter on $(0,1)$ containing $\mathcal I_0$, then $\mathcal F:=\mathcal I\cap \mathcal S\in P_*(\mathcal S)$.
\end{lemma}
\begin{proof}
(1) From the fact that $g(\mathcal F)$ is a proper ideal, it follows that $\mathcal I$ is a cofilter. From the fact that $g(\mathcal F)$ is prime, it follows that $\mathcal I$ is maximal. (Cf.\ \cite[Lemma~4.17, Thm.~4.19]{AJ2001}.)\\
(2) If $S$, $T$ $\in \mathcal F$, then $S\cup T\in \mathcal I$ and $0\in \overline S\subseteq \overline{S\cup T}$. Should $0\notin \overline{\co{(S\cup T)}}$, then there would exist $\eta > 0$ such that $(0,1)= S\cup T \cup (\eta,1) \in \mathcal I$, a contradiction. So $S\cup T\in \mathcal F$.\\
If $S\in \mathcal S$, then also $\co S\in \mathcal S$. So $S\in\mathcal F$ or $\co S\in\mathcal F$ by maximality of $\mathcal I$.
\end{proof}
Consequently, if $P\idealproper \GenK$ is prime, then $\mathcal U:=\{S\subseteq(0,1): e_{\co S}\in P\}$ is an ultrafilter on $(0,1)$ that contains $\{(0,\eta): \eta\in (0,1)\}$.

We recall the definition of the field ${}^\rho\K$ of nonstandard asymptotic numbers \cite{Lightstone, Todorov99, Todorov2004}. Let $\ster\K$ be a fixed nonstandard extension of $\K$. I.e., let $\mathcal U$ be an ultrafilter on an infinite index set $I$ (in this paper, we will be mainly considering the case $I=(0,1)$). Then $\ster\K=\K^I/\Null_{\mathcal U}$, where
\[\Null_{\mathcal U}=\{(x_\eps)_\eps\in \K^I: (\exists S\in\mathcal U)(\forall\eps\in S)(x_\eps=0)\}.\]
For $x$, $y$ $\in\ster\K$, $x\le y$ if and only if $\{\eps\in I: x_\eps\le y_\eps\}\in\mathcal U$, or equivalently, if and only if there exist representatives $(x_\eps)_\eps$, $(y_\eps)_\eps$ such that $x_\eps\le y_\eps$, $\forall\eps\in I$.\\
Let $\rho\in\ster\K$ be a fixed positive infinitesimal ($\ne 0$). Then ${}^\rho\K={\mathcal M}_\rho(\K)/\Null_\rho(\K)$, where
\[{\mathcal M}_\rho(\K)=\{x\in\ster\K: (\exists N\in\N)(\abs{x}\le \rho^{-N})\}\]
and
\[\Null_\rho(\K)=\{x\in\ster\K: (\forall n\in\N)(\abs{x}\le \rho^{n})\}.\]

\begin{thm}\label{thm_quotient_with_max_is_nonstandard_field}
Let $M$ be a maximal ideal of $\GenK$. Let $\mathcal U=\{S\subseteq(0,1): e_{\co S}\in M\}$. Consider the nonstandard field $\ster{\K}$ constructed by means of the ultrafilter $\mathcal U$. Let $\rho\in\ster{\K}$ be the element with representative $(\eps)_\eps$. Then $\rho$ is a positive infinitesimal, so we can consider ${}^\rho\K$. Then there exists a canonical isomorphism between $\GenK/M$ and ${}^\rho\K$. On representatives in $\K^{(0,1)}$, the isomorphism is given by the identity map.\\
In particular, the algebraic, order and topological structure coincide.
\end{thm}
\begin{proof}
By the previous lemma, $\mathcal U$ is an ultrafilter and for each $n\in\N$, $\{\eps\in(0,1): \eps\le 1/n\}\in\mathcal U$, so $\rho$ is infinitesimal.\\
Let $\phi$: $\GenK/M\to {}^\rho\K$ be defined as follows: if $(x_\eps)_\eps\in\K^{(0,1)}$ is a moderate net representing $x\in\GenK$, then $\phi(x+M)= x' + \Null_\rho(\K)$, where $x'\in\ster\K$ is the element with representative $(x_\eps)_\eps$.\\
First, as $(x_\eps)_\eps$ is moderate, it follows that $x'\in{\mathcal M}_\rho(\K)$.
Further, by theorem~\ref{thm_closed_characterization} and the fact that $M$ is closed, $x\in M$ iff
\[(\forall n\in\N)(\exists S\subseteq (0,1))(e_S\in M \,\&\, (\forall\eps\in\co S)(\abs{x_\eps}\le\eps^n))\]
iff $x'\in\Null_\rho(\K)$, by the definition of $\mathcal U$. This shows that $\phi$ is well-defined and injective. To show that $\phi$ is surjective, let $x'\in{\mathcal M}_\rho(\K)$ arbitrarily with representative $(x_\eps)_\eps\in\K^{(0,1)}$. So there exists $N\in\N$ and $S\in\mathcal U$ such that $\abs{x_\eps}\le\eps^{-N}$, for each $\eps\in S$. Let $y_\eps=\begin{cases}x_\eps, &\eps\in S\\0,&\eps\in\co S.\end{cases}$ Then also $(y_\eps)_\eps$ is a representative of $x'$ and $(y_\eps)_\eps$ is a moderate net.\\
Since the algebraic operations are in both cases defined on representatives, $\phi$ is an algebraic isomorphism. Similarly, in both cases, $x\le y$ iff there exist representatives in $\K^{(0,1)}$ for which the inequality holds componentwise, so $\phi$ is an order isomorphism. Further, the valuation $v(x)$ which determines the topology can in both cases be defined as $\sup\{a\in\R: \abs{x}\le\alpha_a\}$ (with $\alpha_a$ the element with representative $(\eps^a)_\eps$), so $\phi$ is a homeomorphism.
\end{proof}
\begin{cor}
For $M$ a maximal ideal of $\GenR$, $\GenR/M$ is spherically complete \cite{Luxemburg, Todorov2004}.
\end{cor}
The fields $\GenK/m$ ($m$ a maximal ideal) have been studied in \cite{Scarpalezos2004} under the name of $m$-reduced generalized constants.

\section{Annihilator ideals}
\begin{prop}\label{prop_annihilators_elementary}
Let $I\ideal\GenK$.
\begin{enumerate}
\item $\Ann(I)=\bigcap_{e_S\in I}e_{\co S}\GenK$
\item $\Ann(I)$ is closed
\item $\Ann(m(I))=\Ann(I)=\Ann(\overline I)$
\item $\Ann(I)\cap I=\{0\}$.
\end{enumerate}
\end{prop}
\begin{proof}
(1) $\subseteq$: if $x\in\Ann(I)$ and $e_S\in I$, then $x e_S=0$, so $x=x e_{\co S}\in e_{\co S}\GenK$.\\
$\supseteq$: by proposition \ref{prop_level_sets}, $\bigcap_{e_S\in I}e_{\co S}\GenK\subseteq\bigcap_{a\in I} \Ann(a)=\Ann(I)$.\\
%let $a\in I$. Considering level sets for $a$, $a=\lim_n a e_{S_n}$, with $a$ invertible w.r.t.\ $e_{S_n}$, so $e_{S_n}\in I$. Then, for $x\in \bigcap_{e_S\in I}e_{\co S}\GenK$, $ax=\lim_n a x e_{S_n}=0$.
(2) By part 1 and theorem \ref{thm_closed_fin_gen}.\\
(3) By part 1, $\Ann(m(I))=\Ann(I)$. By theorem \ref{thm_closed_characterization}, $m(I)=m(\overline I)$, so also $\Ann(\overline I)=\Ann(m(\overline I))=\Ann(m(I))$.\\
(4) If $a\in\Ann(I)\cap I$, then $a^2=0$, so $a=0$.
\end{proof}

The following lemma is a generalization of lemma \ref{lemma_zero_divisors}.
\begin{lemma}\label{lemma_countable_subset_of_annihilator}
Let $a\in\GenK$ and $b_n\in\Ann(a)$, $\forall n\in\N$. Then there exists $S\subseteq(0,1)$ such that $a e_{\co S}=0$ and $b_n e_{S}=0$, $\forall n\in\N$. I.e., $\lspan{b_n: n\in\N}\subseteq e_{\co S}\GenK\subseteq\Ann(a)$.
\end{lemma}
\begin{proof}
Let $(L_m)_{m\in\N}$ be a sequence of level sets for $a$. Let $(b_{n,\eps})_\eps$ be representatives of $b_n$, $\forall n$. As $b_n\in\Ann(a)$, $b_ne_{L_m}=0$, $\forall m$. So
\[(\forall l,m,n\in\N)(\exists \eta_{l,m,n}\in (0,1))(\forall\eps\in L_m\cap (0,\eta_{l,m,n}))(\abs{b_{n,\eps}}\le\eps^l).\]
Let $\eta_m=\min_{l,n\le m}\eta_{l,m,n}$, $\forall m\in\N$. Let $S=\bigcup_{m\in\N} L_m\cap (0,\eta_m)$. For each $m\in\N$, $e_{S} e_{L_m}=e_{L_m}$, so $e_{\co S}e_{L_m}=0$, $\forall m$. By proposition \ref{prop_level_sets}, $e_{\co S}\in\Ann(a)$. Further, let $n\in\N$. We show that $b_n e_S=0$, i.e.,
\[(\forall l\in\N)(\exists \eta >0)(\forall\eps\in S\cap (0,\eta))(\abs{b_{n,\eps}}\le\eps^l).\]
For fixed $n$ and $l$, let $\eta=\min\limits_{{m<\max(n,l)}} \eta_{l,m,n}$. Let $T_1=\bigcup\limits_{m<\max(n,l)} L_m\cap (0,\eta_m)$ and $T_2=\bigcup\limits_{m\ge \max(n,l)} L_m\cap (0,\eta_m)$. For $\eps\in T_1\cap (0,\eta)$, $\abs{b_{n,\eps}}\le\eps^l$ by the choice of $\eta$. For $\eps\in T_2\cap (0,\eta)$, $\abs{b_{n,\eps}}\le\eps^l$ because $\eta_m \le \eta_{l,m,n}$, for each $m\ge \max(n,l)$.
\end{proof}

\begin{lemma}\label{lemma_countable_subset_of_prime}
Let $S\subseteq(0,1)$ with $0\in\overline S$. Consider $e_S\GenK$ as a commutative ring with $e_S$ as its unity. Let $I$ be a proper pseudoprime ideal of $e_S\GenK$. Then for any countably generated ideal $J$ with $J\subseteq I$, there exists $T\subseteq S$ with $e_T\in I\cap\Ann(J)\setminus\{0\}$.
\end{lemma}
\begin{proof}
Notice that a subset of $e_S\GenK$ is an ideal of $e_S\GenK$ iff it is an ideal of $\GenK$. By proposition \ref{prop_level_sets}, $m(J)=m(a\GenK)=\lspan{e_{L_n}:n\in\N}$ for some $a\in\GenK$ and $(L_n)_{n\in\N}$ a sequence of level sets for $a$. As $I$ is proper, $e_S\notin J$, and $0\in \overline {S\setminus L_n}$, $\forall n\in\N$. So we can successively find for each $n\in\N$, two different elements $\eps_n$, $\eps'_n\in (0,1/n)\cap (0,\eps_{n-1})\cap (0,\eps'_{n-1})\cap S\setminus L_n$. Let $T=\{\eps_n:n\in\N\}$, $T'=\{\eps'_n:n\in\N\}$. Then $e_T e_{L_n}=0$, $\forall n\in\N$, so $e_T\in \Ann(m(J))=\Ann(J)$ by proposition \ref{prop_annihilators_elementary}. As $T\subseteq S$, $e_T=e_T e_S\in e_S\GenK$. As $0\in \overline T$, $e_T\ne 0$. Similarly, $e_{T'}\in e_S\GenK\cap\Ann(J)\setminus \{0\}$. As $I$ is a pseudoprime ideal of $e_S\GenK$ and $e_T e_T'=0$ (since $T\cap T'=\emptyset$), either $e_T\in I$ or $e_{T'}\in I$. So either $T$, either $T'$ satisfies the required conditions.
\end{proof}
\begin{cor}
There exists an uncountable family of mutually orthogonal idempotents in $\GenK$.
\end{cor}
\begin{proof}
%(1) Suppose that $I$ is countably generated. By the previous lemma, we find $e_T\in I\cap\Ann(I)\setminus \{0\}=\emptyset$, a contradiction.\\
By Zorn's lemma applied to the set of all sets of mutually orthogonal idempotents contained in a given prime ideal $P$, ordered by inclusion, there exists a maximal set $A$ of mutually orthogonal idempotents contained in $P$. If $A$ would be countable, then there would exist $e_T\in P\cap\Ann(\lspan{A})\setminus \{0\}$, contradicting the maximality of $A$.
\end{proof}

The next theorem shows in particular that, in contrast with the situation in classical Hilbert spaces, for a submodule (=ideal) $M$ of $\GenK$, not necessarily $M^{\perp\perp}=\overline M$, and that $M^\perp=\{0\}$ does not imply that $M$ is (topologically) dense. (The scalar product on $\GenK$ is defined by $\inner{a}{x}=a\bar x$, hence $M^\perp:=\{a\in\GenK: (\forall x\in M)(\inner{a}{x} = 0)\}=\Ann(M)$.)
\begin{thm}\label{thm_double_orthog}
\begin{enumerate}
\item Let $I\ideal\GenK$. Then $\overline I\subseteq\Ann(\Ann(I))$.
\item Let $a\in\GenK$. Then $\Ann(\Ann(a))=\overline{a\GenK}$.
\item Let $I\ideal\GenK$ be countably generated. Then $\Ann(\Ann(I))=\overline I$.
\item Let $I\idealproper\GenK$ be pseudoprime. Then $\Ann(I)=\{0\}$. In particular, $\overline I\subsetneqq \Ann(\Ann(I))=\GenK$.
\end{enumerate}
\end{thm}
\begin{proof}
(1) Let $x\in I$. Then $xy=0$, for each $y\in\Ann(I)$, so $x\in\Ann(\Ann(I))$. The assertion follows from the fact that $\Ann(\Ann(I))$ is closed.\\
(2) Let $(L_n)_{n\in\N}$ be a sequence of level sets for $a$. Let $S\subseteq(0,1)$ such that $e_S\notin a\GenK$. In particular, $e_S\ne e_Se_{L_n}$, so $e_{S\setminus L_n}\ne 0$, for each $n$. So we can recursively find $\eps_n\in(0,\eps_{n-1})\cap(0,1/n)\cap S\setminus L_n$. Calling $T=\{\eps_n : n\in\N\}$, we have $e_Te_S=e_T\ne 0$ and $e_T e_{L_n}=0$, $\forall n$. So $e_T\in\Ann(a)$ by proposition \ref{prop_level_sets}. So $e_S\notin\Ann(\Ann(a))$. By contraposition, $m(\Ann(\Ann(a)))\subseteq a\GenK$. As $\Ann(\Ann(a))$ is closed, $\Ann(\Ann(a))\subseteq \overline{a\GenK}$ by theorem~\ref{thm_closed_characterization}. The converse inclusion follows by part~1.\\
(3) Follows by part~2, proposition \ref{prop_countably_generated_ideals} and proposition \ref{prop_annihilators_elementary}.\\
(4) Suppose $x\in\Ann(I)$, $x\ne 0$. By lemma~\ref{lemma_zero-inv}, there exists $S\subseteq(0,1)$ with $e_S\ne 0$ such that $x$ is invertible w.r.t.\ $S$, so $e_S\in\Ann(I)$. Now there exists $T\subseteq S$ such that both $0\in \overline T$ and $0\in\overline{S\setminus T}$, i.e., $e_T=e_S e_T\ne 0$ and $e_S e_{\co T}\ne 0$. But as $I$ is pseudoprime, either $e_T\in I$ or $e_{\co T}\in I$, in contradiction with $e_S\in\Ann(I)$.
\end{proof}
\begin{cor}
\leavevmode
\begin{enumerate}
\item A proper pseudoprime ideal $I$ of $\GenK$ is not countably generated.
\item Let $I\ideal\GenK$. Then $\Ann(I)$ is not prime.
\end{enumerate}
\end{cor}
\begin{proof}
(2) If $I=\{0\}$, $\Ann(I)=\GenK$ is not prime. So let $I\ne\{0\}$ and suppose that $\Ann(I)$ is prime. By lemma \ref{lemma_zero-inv}(4), there exists $S\subseteq (0,1)$ with $e_S\in I\setminus\{0\}$. Then $\Ann(I)\subseteq e_{\co S}\GenK\subsetneqq\GenK$. As $\Ann(I)$ is closed, $\Ann(I)$ would be maximal by theorem \ref{thm_max_ideals}, hence $\Ann(I)= e_{\co S}\GenK$. This contradicts part~1.
\end{proof}

\begin{prop}
Let $a\in\GenK$.
\begin{enumerate}
\item if $a\GenK$ is closed, then $\Ann(a)$ is a principal ideal.
\item if $a\GenK$ is not closed, then $\Ann(a)$ is not the closure of a countably generated ideal.
%\item if $a\GenK$ is not closed, then $m(\Ann(a))$ is not countably generated.
\end{enumerate}
\end{prop}
\begin{proof}
(1) By theorem \ref{thm_closed_fin_gen}, $a\GenK=e_S\GenK$, for some $S\subseteq(0,1)$, so $\Ann(a)=e_{\co S}\GenK$.\\
(2) Suppose that $\Ann(a)=\overline{\lspan{b_n:n\in\N}}$, for some $b_n\in\GenK$. By lemma \ref{lemma_countable_subset_of_annihilator}, there would exist $S\subseteq (0,1)$ such that $\lspan{b_n:n\in\N}\subseteq e_{S}\GenK\subseteq\Ann(a)$. As $e_S\GenK$ is closed, we would obtain that $\Ann(a)=e_S\GenK$.
By theorem \ref{thm_double_orthog}, $\overline{a\GenK}=\Ann(e_S\GenK)=e_{\co S}\GenK$. But by the corollary to theorem \ref{thm_closed_fin_gen}, $\overline{a\GenK}$ is not principal, a contradiction.
%Let $(L_n)_{n\in\N}$ be a sequence of level sets for $a$. By proposition~\ref{prop_edged_generator}, $(L_n)$ is not stationary, so we find a strictly increasing sequence $(k_n)_{n\in\N}$ and $S_n=L_{k_n}\setminus L_{k_n - 1}$ such that $e_{S_n}\ne 0$, $\forall n$. As $e_S\in\Ann(a)$, $e_S e_{S_n}=0$, for each $n$. As $e_{S_n}\ne 0$, this implies that $e_{S_n\setminus S}\ne 0$. So we can recursively find $\eps_n\in(0,\eps_{n-1})\cap (0,1/n)\cap S_n\setminus S$. Calling $T:=\{\eps_n:n\in\N\}$, we have $S\cap T=\emptyset$, $e_T\ne 0$ and $a e_T=0$. So $e_S + e_T\in\Ann(a)=e_S\GenK$, and $e_T=e_Te_{\co S}=0$, a contradiction.
%(3) By part 2 and the fact that $\Ann(a)=\overline{\Ann(a)}=\overline{m(\Ann(a))}$ (theorem \ref{thm_closed_characterization}).
\end{proof}
\begin{cor}
There exist continuous $\GenK$-linear maps $\GenK\to\GenK$ for which the kernel is not the closure of a countably generated ideal.
\end{cor}
\begin{proof}
%(1) Suppose $\overline{a\GenK}$ is principal. By theorem \ref{thm_closed_fin_gen}, $\overline{a\GenK}=e_S\GenK$, for some $S\subseteq(0,1)$. So $\Ann(\overline{a\GenK})=\Ann(a)=\Ann(e_S)=e_{\co S}\GenK$ by proposition \ref{prop_annihilators_elementary}. By the previous theorem, $\Ann(a)$ is not principal, a contradiction.\\
Let $a\in\GenK$ with $a\GenK$ not closed (cf.\ the remark following theorem \ref{thm_closed_fin_gen}) and $f$: $\GenK\to\GenK$: $f(x)=ax$. Then $\Ker f =\Ann(a)$.
\end{proof}

\section{Projective ideals}
\begin{comment}
If $I =\lspan{x_\lambda : \lambda\in\Lambda}\ideal\GenK$ is projective, then for each $\lambda\in\Lambda$, there exists a $\GenK$-linear map $\phi_\lambda$: $I\to\GenK$ such that for each $x\in I$
\begin{enumerate}
\item $\phi_\lambda(x)=0$, except for finitely many $\lambda\in\Lambda$
\item $x=\sum_{\lambda\in\Lambda} \phi_\lambda(x) x_\lambda$.
\end{enumerate}
$(x_\lambda, \phi_\lambda)_{\lambda\in\Lambda}$ is called a projective base for $I$. In particular, it follows that $I$ is algebraically isomorphic with a direct summand of the free module $\bigoplus_{\lambda\in\Lambda}\GenK$ \cite{Cartan}.
\end{comment}
Using the fact that $\GenK$ is an exchange ring, the structure of the projective ideals of $\GenK$ appears to be straightforward (in contrast with the situation in rings of continuous functions \cite{Brookshear77}).
\begin{lemma}\label{lemma_principal_projective_ideals}
Let $R$ be a commutative ring with $1$. Let $I\ideal R$ be principal. Then $I$ is projective iff $I$ is generated by an idempotent.
\end{lemma}
\begin{proof}
If $I$ is generated by an idempotent, then $I$ is a direct summand of $R$, hence projective.
Conversely, if $I$ is projective, then $I$ is algebraically isomorphic with a direct summand of $R$ \cite[Proof of thm.~1.2.2]{Cartan}, hence generated by an idempotent.
\end{proof}

\begin{thm}\label{thm_projective_ideals}
An ideal $I\ideal\GenK$ is projective iff $I$ is a direct sum of principal ideals generated by idempotents (i.e., $I$ is generated by a family of mutually orthogonal idempotents).
\end{thm}
\begin{proof}
If $I$ is a direct sum of principal ideals generated by idempotents, then $I$ is projective as a direct sum of projective ideals.\\ Conversely, by proposition \ref{prop_GenK_is_exchange}, $\GenK$ is an exchange ring, so a projective ideal $I$ is a direct sum of finitely generated ideals \cite[Thm.~1.7]{McGovern2006}, hence a direct sum of principal ideals by lemma \ref{lemma_principal_ideals}. Each of these principal ideals is projective (as a direct summand of a projective ideal), hence generated by an idempotent by lemma \ref{lemma_principal_projective_ideals}.
\end{proof}
%in fact, every projective $\GenK$-module is a direct sum of ideals generated by idempotents \cite{Warfield72}

\begin{cor}
A projective ideal $I\ideal\GenK$ is pure.
\end{cor}
%This corollary may hold in any exchange ring \cite[Thm.~1]{Warfield72}.
\begin{rem}
By the corollary to lemma \ref{lemma_countable_subset_of_prime}, there exist uncountably generated projective ideals in $\GenK$.
\end{rem}

\begin{prop}
A countably generated pure ideal is projective.
\end{prop}
\begin{proof}
Let $I$ be a countably generated pure ideal. By proposition \ref{prop_level_sets}, $I=\lspan{e_{L_n}: n\in\N}$, where $(L_n)_{n\in\N}$ is a sequence of level sets for some $a\in\GenK$. Let $S_n=L_n\setminus L_{n-1}$, $\forall n\in\N$. Then $I=\lspan{e_{S_n}: n\in\N}$. As $S_n\cap S_m=\emptyset$ for $n\ne m$, also $e_{S_n}e_{S_m}=0$ if $n\ne m$. So $I=\bigoplus_{n\in\N}e_{S_n}\GenK$ is projective by theorem \ref{thm_projective_ideals}.
\end{proof}

\begin{prop}
A proper pseudoprime ideal of $\GenK$ is not projective.
\end{prop}
\begin{proof}
Suppose that $P$ is a proper pseudoprime ideal generated by a family $\mathcal E$ of mutually orthogonal idempotents. Let $(e_{S_n})_{n\in\N}$ be a family of different elements of $\mathcal E$ ($P$ is not finitely generated by the corollary to theorem \ref{thm_double_orthog}). For $n\in\N$, $e_{S_1} e_{S_n}=\cdots = e_{S_{n-1}}e_{S_n}=0$, so there exists $\eta\in (0,1)$ such that $S_1\cap S_n \cap (0,\eta)=\cdots = S_{n-1}\cap S_n\cap (0,\eta)=\emptyset$. As $e_{S_n}=e_{S_n\cap (0,\eta)}$, we may suppose that $S_m\cap S_n=\emptyset$, $\forall m\ne n$. As w.l.o.g.\ for each $n$, $e_{S_n}\ne 0$, i.e., $0\in \overline{S_n}$, we can write $S_n=T_n\cup U_n$ with $0\in \overline{T_n}$, $0\in \overline{U_n}$ and $T_n\cap U_n=\emptyset$. Then consider $T=\bigcup_{n\in\N} T_n$ and $U=\bigcup_{n\in\N} U_n$. As $T\cap U=\emptyset$, $e_Te_U= 0$, hence $e_T\in P$ or $e_U\in P$.
By symmetry, we may suppose that $e_T\in P$. Then $e_T=a_1 e_{V_1}+\cdots + a_m e_{V_m}$, for some $m\in\N$, $a_j\in\GenK$ and $e_{V_j}\in\mathcal E$. Let $n\in\N$. If $e_{S_n}\notin \{e_{V_1},\dots, e_{V_m}\}$, then $e_{T_n} = e_T e_{S_n} = 0$ by orthogonality, a contradiction. So $\{e_{S_n}:n\in\N\}\subseteq \{e_{V_1},\dots, e_{V_m}\}$, a contradiction.
\end{proof}
In particular, there exist pure ideals of $\GenK$ that are not projective.

\section{$\GenK$-linear maps}
In analogy with the classical Hahn-Banach extension property, one could ask if, for (e.g.) a Banach $\GenK$-module $\Gen$ \cite{Garetto2005, Garetto2006}, a submodule $M$ of $\Gen$ and a continuous $\GenK$-linear functional $\phi$: $M\to\GenK$, there always exists an extension of $\phi$ to a continuous $\GenK$-linear functional $\Gen\to\GenK$.
The following theorem shows that (under some set-theoretic assumption) this formulation of the Hahn-Banach extension property does not hold for the case where $\Gen=\GenK$. Since every Banach $\GenK$-module $\Gen_E$ constructed by means of a classical Banach space $E$ \cite{Garetto2005} contains $\GenK$ as a topological submodule, this formulation of the Hahn-Banach extension property also does not hold for any such $\GenK$-module. For a more restricted version of a Hahn-Banach extension property on Banach $\GenK$-modules, see however \cite{Mayerhofer2006} (the restriction there is that the obtained extension is merely $L$-linear for a certain subfield $L$ of $\GenK$). We also like to mention that the Hahn-Banach extension property holds for nonarchimedean normed linear spaces over ${}^\rho\K$ \cite{Luxemburg}.

\begin{thm}\label{thm_Hahn-Banach}
Assume that $2^{\aleph_0}<2^{\aleph_1}$ (e.g., assume the continuum hypothesis).
Then there exists an ideal $I\idealproper\GenK$ and a continuous $\GenK$-linear map $\phi$: $I\to\GenK$ that cannot be extended to a $\GenK$-linear map $\psi$: $\GenK\to\GenK$.
\end{thm}
\begin{proof}
Let $S=\{\frac{1}{n}:n\in\N\} \subseteq(0,1)$. Let $J$ be a prime ideal of the ring $e_S\GenK$.\\%The maximality implies that for each $T\subseteq S$, either $e_T\in J$ or $e_{\co T}\in J$.\\
By transfinite recursion, we define for each countable non-limit ordinal $\zeta$ a subset $T_\zeta\subseteq S$ as follows. Let $T_\xi$ be defined for each non-limit ordinal $\xi < \zeta$ such that $e_{T_\xi}\in J$. Let $I_\zeta= \lspan{e_{T_\xi}: \xi < \zeta}\subseteq J$. Since $\zeta$ is countable, we find by lemma \ref{lemma_countable_subset_of_prime} $T_\zeta\subseteq S$ with $e_{T_\zeta}\in J\cap\Ann(I_\zeta)\setminus\{0\}$. In particular, $e_{T_\zeta}\in J\setminus I_\zeta$ and $e_{T_\xi} e_{T_\zeta}=0$, for each non-limit ordinal $\xi<\zeta$. Also for limit ordinals $\zeta$, we denote $I_\zeta=\lspan{e_{T_\xi}: \xi < \zeta}$.\\
%By transfinite recursion, we can define for each ordinal number $\zeta$ a subset ${\mathcal F}_\zeta\subseteq\mathcal F$ such that ${\mathcal F}_{\zeta+1}={\mathcal F}_\zeta\cup\{T_\zeta\}$ with some $e_{T_\zeta}\notin g({\mathcal F}_\zeta)$, unless $g({\mathcal F}_\zeta)=J$.\\
%Suppose that for some countable $\zeta$, $g({\mathcal F}_\zeta)= J$. Then we can obtain (as in the proof of proposition \ref{prop_countably_generated_ideals}) an increasing sequence $(S_n)_{n\in\N}$ of subsets of $(0,1)$ with $J=g(\{S_n: n\in\N\})$. Then take for each $n\in\N$, two different elements $\eps_n,\eps'_n\in (0,1/n)\cap S\setminus S_n$, and let $T=\{\eps_n:n\in\N\}$, $T'=\{\eps'_n:n\in\N\}$. Then $e_Te_{S_n}=0$, for all $n$, so $e_T\notin J$. Similarly, $e_{T'}\notin J$. By maximality of $J$, either $e_T$, either $e_{\co T}$ (hence also $e_{T'}$) belongs to $J$, a contradiction.\\
Let $\omega^+$ be the least ordinal of uncountable cardinality. We show that there exist continuous linear functionals on $I_{\omega^+}$ that cannot be extended to a linear map $\GenK\to\GenK$.\\
Let $\zeta<\omega^+$ be an ordinal. We show that any $\phi$: $I_\zeta\to\GenK$ with $\abs{\phi(x)}\le\abs{x}$, $\forall x\in I_\zeta$ can be extended to a linear map $\psi$: $I_{\zeta+1}=I_{\zeta} + e_{T_\zeta}\GenK\to \GenK$ with $\abs{\psi(y)}\le\abs{y}$, $\forall y\in I_{\zeta+1}$.
Define $\psi(e_{T_\zeta})$ as the element with representative $(c_\eps)_\eps$ where $\abs{c_\eps}\le 1$, $\forall\eps\in T_\zeta$ and $c_\eps=0$, $\forall\eps\in \co T_\zeta$. By definition, $\abs[]{\psi(e_{T_\zeta})}\le e_{T_\zeta}$.\\
We show that the map $\psi$: $I_{\zeta+1}\to\GenK$: $\psi(x+\lambda e_{T_\zeta}) = \phi(x)+\lambda\psi(e_{T_\zeta})$ ($x\in I_\zeta$, $\lambda\in\GenK$) is well-defined. For each $x\in I_{\zeta}$, $x e_{T_\zeta}=0$. So if $x+\lambda e_{T_\zeta}=0$, multiplication by $e_{T_\zeta}$ shows that $x= \lambda e_{T_\zeta}=0$. Then also $\abs{\phi(x)+\lambda\psi(e_{T_\zeta})}\le\abs\lambda e_{T_\zeta}=0$, and $\phi(x)+\lambda\psi(e_{T_\zeta})=0$.\\
Further, for $y=x+\lambda e_{T_\zeta}\in I_{\zeta+1}$ ($x\in I_\zeta$, $\lambda\in\GenK$),
\[
\abs{\psi(y)}\le \abs{\phi(x)}+\abs{\lambda \psi(e_{T_\zeta})}
\le \abs{x}+\abs{\lambda e_{T_\zeta}}
=\abs{x+\lambda e_{T_\zeta}}(e_{\co T_\zeta} + e_{T_\zeta})=\abs{y}.
\]
Also for a limit ordinal $\zeta$, it is clear that $\psi$: $I_\zeta\to\GenK$: $\psi(x)=\phi_\xi(x)$, if $x\in I_\xi$ defines a linear map satisfying $\abs{\psi(x)}\le\abs{x}$, as soon as the same holds for all $\phi_\xi$, $\xi<\zeta$.\\
Now for each non-limit ordinal $\zeta$, we have at least two choices for defining $\psi(e_{T_\zeta})$, e.g.\ $\psi(e_{T_\zeta})=0$ or $\psi(e_{T_\zeta})=e_{T_\zeta}$. Denoting $\Lambda$ the set of all countable non-limit ordinals and $I_{\omega^+}'$ the set of all continous linear functionals on $I_{\omega^+}$, we find a surjective map from $I_{\omega^+}'$ onto the set $\{(\phi(e_{T_\zeta}))_{\zeta\in \Lambda}: \phi\in I_{\omega^+}'\}$, which can be surjectively mapped onto the set $\{0,1\}^\Lambda$. As the cardinality of $\Lambda$ is $\aleph_1$, the cardinality of $I_{\omega^+}'$ is at least $2^{\aleph_1}$.
%, by choosing other representatives of $\phi(e_{S_n})$. E.g., we can change the values $c_\eps$ for $\eps\in T_\zeta\setminus \bigcup_n S_n$, if $0\in\overline{T_\zeta\setminus \bigcup_n S_n}$; otherwise, as $e_{T_\zeta}\notin g({\mathcal F}_\zeta)$, $(S_n)_{n\in\N}$ cannot be stationary, so we can change the values $c_{\eps_n}$ for $\eps_n\in (0,1/n)\cap S_n\setminus S_{n-1}$ (always ensuring that $\abs{c_{\eps_n}}\le 1$). The cardinal number of the set of non-limit ordinals $<\omega^+$ is $\aleph_1$, so there are at least $2^{\aleph_1}$ continuous $\GenK$-linear maps $\phi$: $I\to\GenK$.
If each $\phi\in I_{\omega^+}'$ could be extended to a linear map $\psi$: $e_S\GenK\to\GenK$, we would obtain $2^{\aleph_1}$ different continuous $\GenK$-linear maps $\psi$: $e_S\GenK\to\GenK$. Now $\psi$ is completely determined by the number $\psi(e_S)\in e_S\GenK$, so by one sequence of rational numbers (since any element in $e_S\GenK$ has a representative consisting of $0$ for $\eps\in \co S$, and of $q_1+iq_2$, $q_j\in\Q$, for $\eps\in S$). The cardinality of the set of such sequences is $2^{\aleph_0}<2^{\aleph_1}$, a contradiction.
\end{proof}

\section{Examples}
An ultrafilter $\mathcal U$ is called $\delta$-stable iff the following property holds:\\
let for each $n\in\N$, $J_n\in\mathcal U$; then there exists $J\in\mathcal U$ such that, for each $n\in\N$, $J\setminus J_n$ is finite.
\begin{thm}(\cite{Choquet68}, \cite[\S~7.1]{S&L})
Assume the continuum hypothesis. Then there exists a $\delta$-stable free ultrafilter on $\N$.
\end{thm}

This theorem has the following corollary.
\begin{thm}
Assume the continuum hypothesis. Then there exists $\mathcal F\in P_*(\mathcal S)$ such that $g(\mathcal F)$ is closed (and hence, maximal).
\end{thm}
\begin{proof}
Let $\mathcal U$ be a $\delta$-stable free ultrafilter on $\N$. For $S\subseteq (0,1)$, let $1/S=\{1/\eps:\eps\in S\}$.
Let $\mathcal F=\{S\in\mathcal S: \N\cap 1/S\notin \mathcal U\}$.\\
We show that $\mathcal F\in P_*(\mathcal S)$.\\
Let $S$, $T\in \mathcal F$, so there exist $J_1$, $J_2$ $\in \mathcal U$ such that $\N\cap 1/S=\N\setminus J_1$, $\N\cap 1/T=\N\setminus J_2$. Then $\N\cap 1/(S\cup T)=\N\setminus(J_1\cap J_2)\notin\mathcal U$. Further, as $0\in\overline S$, also $0\in\overline{S\cup T}$. Should $0\notin\overline{\co {(S\cup T)}}$, then $(0,\eta)\subseteq S\cup T$, for some $\eta\in (0,1)$, and $J_1\cap J_2$ would be a finite element of the free ultrafilter $\mathcal U$, a contradiction. So $S\cup T\in \mathcal F$.\\
Let $S\in\mathcal S$. Call $J_S=\N\cap 1/S$. If $J_S\notin\mathcal U$, then $S\in\mathcal F$. Otherwise, $\N\cap1/\co S=\N\setminus J_S\notin\mathcal U$, so $\co S\in\mathcal F$.\\
Let $a\in\overline{g(\mathcal F)}$. We show that $a\in g(\mathcal F)$.\\
Let $(L_n)_{n\in\N}$ be a sequence of level sets for $a$.
By theorem~\ref{thm_closed_characterization}, $e_{L_n}\in g(\mathcal F)$, $\forall n$. If $a\ne 0$, there exists $N\in\N$ such that for each $n\in\N$ with $n\ge N$, $L_n\in\mathcal S$, so by \cite[Rem.~4.18]{AJ2001}, $L_n\in\mathcal F$, i.e., $\N\cap 1/L_n=\N\setminus J_n$, for some $J_n\in\mathcal U$. By the $\delta$-stability, there exists $J\in\mathcal U$ such that $J\setminus J_n$ is finite, for each $n\ge N$. Let $T=(0,1)\setminus\{1/n: n\in J\}$. Then $\N\cap 1/T=\N\setminus J$, so $T\in \mathcal F$. As $J\setminus J_n$ is finite, $L_n\cap (0,\eta)\subseteq T$, for some $\eta>0$ (depending on $n$), and $e_{L_n}e_T= e_{L_n\cap T} = e_{L_n}$, for each $n\ge N$, and $a =\lim_n a e_{L_n} = \lim_n a e_{L_n} e_T = a e_T\in g(\mathcal F)$.
\end{proof}

In $\Cnt{}(X)$, there can exist prime ideals that are not $z$-ideals; the same holds in $\GenK$ (the construction of an example is however completely different \cite[2G.1]{GJ}).
\begin{ex}\label{ex_non-z-prime}
There exist infinitely many prime ideals $P_2\supset P_3\supset\cdots\supset P_m\supset\cdots$ of $\GenK$ that are not $z$-ideals. Moreover, $\zclosure{(P_n)}=\zclosure{(P_m)}$, $\forall m,n$.
\end{ex}
\begin{proof}
Let $S_n\in\mathcal S_1$, for each $n\in\N$, with $S_n\cap S_m=\emptyset$ if $n\ne m$.\\
For $m\in\N$, let $\beta_m\in\GenK$ be defined as follows on representatives:%(cf.~\cite[lemma 4.3]{AJOS2006}):
\[
\begin{cases}
(\beta_m)_\eps = \eps^{n^m},& \eps\in S_n, n\in\N\\
(\beta_m)_\eps= 0, &\text{otherwise}.
\end{cases}
\]
Consider the set \[A= \{x\in \GenK: (\exists N\in\N)(\forall n\ge N) (S_n\in\Inv(x))\}.\]
If $x$, $y$ $\in A$, then there exists $N\in\N$ such that for $n\ge N$, $S_n\in\Inv(x)\cap\Inv(y)\subseteq\Inv(xy)$, so $xy\in A$. Hence $A$ is closed under multiplication. Further $0\notin A$, so by proposition \ref{prop_prime_constructor}, there exists a prime ideal $P$ such that $P\cap A=\emptyset$. %Consider $\mathcal F=\{S\in\mathcal S: e_S\in P\}\in P_*(\mathcal S)$.
We show that for each $m > 1$, $\beta_{m-1}\notin \rad{m(P) + \beta_m\GenK}$.\\
Suppose that there exist $x\in m(P)$, $y\in\GenK$ and $n\in\N$ such that $x + y \beta_m = \beta_{m-1}^n$.
As $x\in m(P)$, $x=xe_T$ for some $e_T\in P$.
So $e_T\notin A$, and for each $N\in\N$, there exists $k\ge N$ such that $S_k\notin\Inv(e_T)$. By lemma~\ref{lemma_zero-inv}(9), $0\in\overline{S_k\setminus T}$, i.e., $e_{U_k}\ne 0$ with $U_k=S_k\setminus T$. So $y \beta_m e_{U_k} = \beta_{m-1}^n e_{U_k}$. As $U_k\subseteq S_k$, $\beta_m e_{U_k}= \alpha^{k^m} e_{U_k}$, and $y e_{U_k} = \alpha^{n k^{m-1}-k^m} e_{U_k}=\alpha^{-k^{m-1}(k-n)} e_{U_k}$. As $n$, $m$ are fixed and $k$ can be chosen arbitrary large, this contradicts the moderateness of $y$.\\
By theorem~\ref{thm_prime_characterization}, $P_m:=\rad{m(P) + \beta_m\GenK}$ is a prime ideal for each $m>1$. Since $\beta_m\in\beta_{m-1}\GenK$ and $\beta_{m-1}\in P_{m-1}\setminus P_m$, we have $P_m\subsetneqq P_{m-1}$, for all $m>1$.\\
Finally, by proposition~\ref{prop_level_sets}, for each $m\in\N$,
\[\Inv(\beta_m)=\{S\in {\mathcal S}_1: (\exists N\in\N)(\exists \eta >0)(S\cap (0,\eta)\subseteq S_1\cup\cdots\cup S_N)\},\]
so for each $n$, $m$, $\Inv(\beta_m)=\Inv(\beta_n)$, hence $\zclosure{(\beta_m\GenK)}=\zclosure{(\beta_n\GenK)}$. By propositions \ref{prop_z-closure}(2) and \ref{prop_z_part}(1), $\zclosure{(P_m)}=\zclosure{(m(P) + \beta_m\GenK)}=\zclosure{m(P)} + \zclosure{(\beta_m\GenK)}$.
Hence $\zclosure{(P_m)}=\zclosure{(P_n)}$, for each $m$, $n$, and $P_m\subsetneqq \zclosure{(P_m)}$, for $m>1$.
\end{proof}
As a result, we can answer a question posed in~\cite{AJ2001}. The Krull dimension of $\GenK$ is defined as the supremum of $n\in\N$ such that there exist $P_0$, $P_1$, \dots, $P_n$ prime ideals of $\GenK$ with $P_0\subset P_1\subset\cdots\subset P_n$.
\begin{thm}
The Krull dimension of $\GenK$ is infinite.
\end{thm}

\begin{ex}\label{ex_z-prime}
There exists a prime $z$-ideal of $\GenK$ that is neither a minimal nor a maximal prime ideal.
\end{ex}
\begin{proof}
Let $S_n, T_n\in\mathcal S_1$, for each $n\in\N$, with $T_n\cap T_m=S_n\cap S_m=\emptyset$ if $n\ne m$ and $\bigcup_{n\in\N} S_n= (0,1)$.
Further, let $0\in\overline{S_n\cap T_m}$, $\forall n,m\in\N$, let $(1/n,1)%\cap \bigcup_{n\in\N}S_n
\subseteq T_1\cup\cdots \cup T_n$, $\forall n\in\N$.
% and let $\bigcup_{n\in\N} T_n\subseteq \bigcup_{n\in\N} S_n$.
The sequence $(T_n)_{n\in\N}$ can be constructed starting from the sequence $(S_n)_{n\in\N}$, using the property that for $S\in\mathcal S_1$ and $\eta>0$, it is always possible to find $T\subseteq S$ with $(\eta,1)\cap S\subseteq T$, $T\in\mathcal S_1$ and $S\setminus T\in\mathcal S_1$.
Define $\beta$, $\gamma$ as follows on representatives:
\begin{align*}
\begin{cases}
\beta_\eps = \eps^n,&\eps\in S_n, n\in\N\\
\beta_\eps = 0,&\text{otherwise},
\end{cases}
&
&
\begin{cases}
\gamma_\eps = \eps^{n+m},&\eps\in S_n\cap T_m, n,m\in\N\\
\gamma_\eps = 0,&\text{otherwise}.
\end{cases}
\end{align*}
Further, call $\widetilde S_n=S_1\cup\cdots\cup S_n$ and $\widetilde T_n=T_1\cup\cdots\cup T_n$, for $n\in\N$. Let
\[
\mathcal F_0 := \{S_k:k\in\N\}\cup \big\{\bigcup_{n\in\N} S_n \cap \widetilde T_{m_n}: (m_n)_{n\in\N}\in\N^{\N}\big\}
%\cup \big\{(0,1)\setminus\bigcup_{n\in\N} S_n\big\}
\subseteq \mathcal S.
\]
%(In case $0\notin\overline{(0,1)\setminus \bigcup_{n\in\N} S_n}$, it must be excluded from $\mathcal F_0$.)\\
We show that $g(\mathcal F_0)\cap\{\gamma^n:n\in\N\}=\emptyset$.\\
Let $x\in g(\mathcal F_0)$. There exist $k\in\N$, $(m_n)_{n\in\N}\in \N^{\N}$ such that $x=xe_S$ with
\[
S=\widetilde S_k \cup \bigcup_{n\in\N}(S_n \cap \widetilde T_{m_n})
%\cup \big((0,1)\setminus \bigcup_{n\in\N} S_n\big)
.
\]
In particular, for $l>k$ and $m>m_l$, $S\cap (S_l\cap T_m) =\emptyset$, so $e_{S_l\cap T_m}\in Z(x)$. But $S_l\cap T_m\in\Inv(\gamma^n)$. So $x\ne \gamma^n$, $\forall n\in\N$.\\
By proposition \ref{prop_prime_constructor}, there exists a prime ideal $P\supseteq g(\mathcal F_0)$ with $\gamma\notin P$.\\
We show that $\beta\notin \zclosure{(m(P) + \gamma\GenK)}$.\\
Suppose that $x_1 e_S + x_2 \gamma=y$, for some $x_1$, $x_2\in\GenK$, $e_S\in P$ and $y\in\GenK$ with $Z(y)=Z(\beta)$. As $1\notin P$, $S\cup \bigcup_{n\in\N} (S_n\cap \widetilde T_{m_n})%\cup (0,1)\setminus (\bigcup_{n\in\N} S_n)
\ne (0,1)$, for any $(m_n)_{n\in\N}$. So $S$ cannot have the property that $(\forall n\in\N)(\exists m\in\N)(S_n\setminus \widetilde T_m\subseteq S)$. I.e., $(\exists N\in\N)(\forall m\in\N)(S_N\setminus\widetilde T_m\not\subseteq S)$. For any $m\in\N$, let $\eps_m\in (S_N\setminus\widetilde T_m)\setminus S$. Let $T=\{\eps_n:n\in\N\}$. By definition of $S_n$ and $T_n$, $\eps_n\le 1/n$, so $0\in\overline T$. Further $T\cap S=\emptyset$, so $e_S e_T = 0$; for each $m$, $T\cap \widetilde T_m$ is finite, so $\gamma e_T = 0$. Hence $y e_T = 0$. But $T\subseteq S_N$, so $T\notin Z(y)=Z(\beta)$, a contradiction.\\
So $\widetilde P:=\zclosure{(m(P)+\gamma\GenK)}$ is a prime $z$-ideal with $m(P)\subseteq \widetilde P\subseteq \overline{m(P)}$. As $\beta = \lim_n\beta e_{\widetilde S_n}$, $\beta\in \overline {m(P)}\setminus \widetilde P$ and $\gamma\in \widetilde P\setminus m(P)$. Hence $m(P)\subsetneqq \widetilde P\subsetneqq \overline{m(P)}$.
\end{proof}
%Question: can this example be extended to an arbitrary large number of $z$-ideals that contain the same $g(\mathcal F)$, $\mathcal F\in P_*(\mathcal S)$?


\begin{thebibliography}{10}
\bibitem{AJ2001}
J.~Aragona, S.O.~Juriaans,
{\it Some structural properties of the topological ring of Colombeau's generalized numbers},
Comm.~Alg., 29(5), 2201--2230 (2001).

\bibitem{AJOS2006}
J.~Aragona, S.O.~Juriaans, O.R.B.~Oliveira, D.~Scarpalezos,
{\it Algebraic and Geometric Theory of the Topological Ring of Colombeau Generalized Functions},
Proc.\ Edin.\ Math.\ Soc., to appear.

\bibitem{Azarpanah07}
F.~Azarpanah, R.~Mohamadian,
{\it $\sqrt{z}$-ideals and $\sqrt{z^\circ}$-ideals in $\Cnt{}(X)$},
Acta Math.\ Sinica, 23(6), 989-–996 (2007).

\bibitem{Banaschewski2004}
B.~Banaschewski,
{\it Ring theory and pointfree topology},
Topology and its Applications 137 (2004) 21-–37.

\bibitem{Biagioni1990}
H.~Biagioni,
{\em A nonlinear theory of generalized functions},
Lecture Notes in Math.\ vol.~1421, Springer, 1990.

\bibitem{BKW}
A.~Bigard, K.~Keimel, S.~Wolfenstein,
{\it Groupes et anneaux r\'eticul\'es},
Springer, Lecture Notes in Mathematics 608, 1977.

\bibitem{Borceux83}
F.~Borceux, G.~Van den Bossche,
{\it Algebra in a localic topos with applications to ring theory},
Springer, Lecture Notes in Mathematics 1038, 1980.

\bibitem{Brookshear77}
J.~Brookshear,
{\it Projective ideals in rings of continuous functions},
Pacific J.\ Math., 71(2), 313--333, 1977.

\bibitem{Cartan}
H.~Cartan, S.~Eilenberg,
{\it Homological algebra},
Princeton University Press, 1956.

\bibitem{Choquet68}
G.~Choquet,
{\it Deux classes remarquables d'ultrafiltres sur $\N$},
Bull.\ Sci.\ Math. 92(2), 143--153 (1968).

\bibitem{Garetto2005}
C.~Garetto,
{\it Topological structures in Colombeau algebras: topological $\widetilde\C$-modules and duality theory},
Acta Appl.~Math., 88/1, 81-123 (2005).

\bibitem{Garetto2006}
C.~Garetto,
\newblock {\em Closed graph and open mapping theorems for topological $\widetilde{\C}$-modules and applications},
\newblock {arXiv:math.FA/0608087(v2)}, 2006.

%\bibitem{Hilbert_preprint}
%C.~Garetto, H.~Vernaeve,
%{\it Hilbert $\GenC$-modules},
%preprint.

\bibitem{GJ}
L.~Gillman, M.~Jerison,
{\it Rings of Continuous Functions},
Springer, 1960.

\bibitem{GK60}
L.~Gillman, C.W.~Kohls,
{\it Convex and pseudoprime ideals in rings
of continuous functions},
Math.~Z.\ 72, 399--409 (1960).

\bibitem{GKOS}
M.~Grosser, M.~Kunzinger, M.~Oberguggenberger, R.~Steinbauer,
{\it Geometric theory of generalized functions with applications to general relativity},
Kluwer, 2001.

\bibitem{Huijsmans74}
C.~Huijsmans,
{\it Some analogies between commutative rings, Riesz spaces and distributive lattices},
Indag.\ Math.\ 36, 132--147 (1974).

\bibitem{Lambek73}
J.~Lambek, G.~Michler,
{\it The torsion theory at a prime ideal of a right Noetherian ring},
J.~Algebra 25, 364--389 (1973).

\bibitem{Larson88}
S.~Larson,
{\it Pseudoprime $l$-ideals in a class of $f$-rings},
Proc.\ Amer.\ Math.\ Soc.\ 104, 685--692 (1988).

\bibitem{Lightstone}
A.H.~Lightstone, A.~Robinson,
{\it Nonarchimedean fields and asymptotic expansions},
North-Holland, 1975.

\bibitem{Luxemburg}
W.~Luxemburg,
{\it On a class of valuation fields introduced by Robinson},
Israel J.\ Math.\ 25, 189--201 (1976).

\bibitem{Mason73}
G.~Mason,
{\it $z$-ideals and prime ideals},
J.~Alg.\ 26, 280--297 (1973).

\bibitem{Mason80}
G.~Mason,
{\it Prime $z$-ideals of ${\mathcal C}(X)$ and related rings},
Canad.\ Math.\ Bull.\ 23, 437--443 (1980).

\bibitem{Matsumura}
H.~Matsumura,
{\it Commutative ring theory},
Cambridge University Press (English translation) 1986.

\bibitem{Mayerhofer2006}
E.~Mayerhofer,
{\it Spherical Completeness of the ring of Colombeau generalized numbers}, Bull.\ Inst.\ Math.\ Acad.\ Sinica, to appear.

\bibitem{McGovern2006}
W.~McGovern,
{\it Neat rings},
J.\ Pure Appl.\ Alg., 205, 243--265 (2006).

%\bibitem{internal_preprint}
%M.~Oberguggenberger, H.~Vernaeve,
%{\it Internal sets},
%preprint.

\bibitem{Scarpalezos93}
D.~Scarpal\'{e}zos,
{\it Topologies dans les espaces de nouvelles fonctions
  g\'{e}n\'{e}ralis\'{e}es de {C}olombeau. ${\widetilde{\C}}$-modules
  topologiques},
\newblock Universit\'{e} Paris 7, 1993.

\bibitem{Scarpalezos2004}
D.\ Scarpal\'{e}zos,
\emph{$m$-reduced asymptotic constructions, fields and generalized functions},
in \emph{Nonlinear Algebraic Analysis and Applications, Proceedings of the International Conference on Generalized Functions (ICGF 2000)}, A.\ Delcroix, M.\ Hasler, J.-A.\ Marti and V.\ Valmorin, eds., Cambridge Scientific Publishers (2004), 305--316.

\bibitem{S&L}
K.~Stroyan, W.~Luxemburg,
{\it Introduction to the theory of infinitesimals},
Academic Press, 1976.

\bibitem{Subramanian67}
H.~Subramanian,
{\it $l$-prime ideals in $f$-rings},
Bull.\ Soc.\ Math.\ France 95 (1967), 193--203.

\bibitem{Todorov99}
T.~Todorov,
\newblock {\em Pointwise values and fundamental theorem in the algebra of asymptotic functions},
\newblock in  {\em Nonlinear Theory of Generalized Functions}, M.~Grosser, G.~H{\"o}rmann, M.~Kunzinger, and M.~Oberguggenberger,
  eds., Research Notes in Math.\ vol.\ 401, Chapman \& Hall /
  CRC (1999), 369--382.

\bibitem{Todorov2004}
T.~Todorov,
{\it Hahn field representation of A.\ Robinson's asymptotic numbers},
in \emph{Nonlinear Algebraic Analysis and Applications, Proceedings of the International Conference on Generalized Functions (ICGF 2000)}, A.\ Delcroix, M.\ Hasler, J.-A.\ Marti and V.\ Valmorin, eds., Cambridge Scientific Publishers (2004), 357--374.
%arXiv:math.AC/0601722.

%\bibitem{Warfield72}
%R.~Warfield,
%{\it Exchange rings and decompositions of modules},
%Math.\ Ann.\ 199, 31--36 (1972).
\end{thebibliography}
\end{document}